\documentclass[10pt]{article}   
\usepackage{amsmath,nccmath,amssymb,amsfonts,amsthm,mathtools,xfrac,enumitem}  
\usepackage[latin1]{inputenc}
\usepackage[T1]{fontenc}
\usepackage{amscd}  
\usepackage{xr-hyper} 
\usepackage{hyperref}     

\numberwithin{equation}{section}  
\newcommand{\lap}{\Delta}
\newcommand{\K}{\bar{K}}

\newcommand{\N}{\mathbb{N}}
\newcommand{\R}{\mathbb{R}}
\newcommand{\e}{\;\exists \;}
\newcommand{\fa}{\;\forall\;}
\renewcommand{\{}{\lbrace}
\renewcommand{\}}{\rbrace}
\newcommand{\var}{\varphi}

\newcommand\Item[1][]{%
  \ifx\relax#1\relax  \item \else \item[#1] \fi
  \abovedisplayskip=-0pt\abovedisplayshortskip=0pt~\vspace*{-\baselineskip}}

\newtheorem{thm}{Theorem}[section] 
\newtheorem{thm*}{Theorem}
\newtheorem{proposition}[thm]{Proposition}
\newtheorem{lemma}[thm]{Lemma}
\newtheorem{remark}[thm]{Remark}

\newtheorem{Condition}[thm]{Condition}
\newtheorem{definition}[thm]{Definition}

\title{Prescribing scalar curvatures: \\non compactness  versus critical points at infinity
\\
\
\\
Published in Geometric Flows} 
 
\author{Martin Mayer \\ 
\small Scuola Normale Superiore, Pisa, ITALY, 
\small  martin.mayer@sns.it
}

\begin{document}

\maketitle

\begin{abstract} 
 We illustrate an example of a generic, positive function $K$ on a Riemannian manifold to be conformally prescribed  as the scalar curvature, for which the corresponding Yamabe type
 $L^{2}$-gradient flow exhibits  non compact flow lines,  while a slight modification of  it is compact.  
\end{abstract}
\begin{center}
\small{\it Key Words:  Conformal geometry, scalar curvature,  critical points at infinity, geometric flows
\\
Subject classification numbers: 35B33 35R01 53A30 53C44  
}
\end{center}

\tableofcontents

\section{Introduction}\label{sec_introduction} 
 Within the setting of conformally prescribing the scalar curvature on a Riemannian manifold and in the context of the calculus of variations, i.e. by considering an associated energy functional,
we shall illustrate in a very particular case the difference of non compact flow lines of a given gradient flow to critical points at infinity, as we have discussed in \cite{MM3}, namely showing, that
 the volume preserving $L^{2}$- gradient flow \eqref{Yamabe_flow}, which is a natural analogon to the Yamabe flow and was studied in \cite{may-cv}, exhibits  one specific, single bubbling non compactness for exactly one energetic value of the variationally  associated  prescribed scalar curvature functional, while a suitable modification of this flow  eliminates any non compactness. And, as we shall see, the same holds true  for the  strong gradient type flow \eqref{gradient_flow} modified  to preserve the conformal volume just like \eqref{Yamabe_flow}.   Hence as a take away  those   non compact flow lines do not induce critical points at infinity, cf.  \cite{MM3}, i.e. these flows  lead to variationally unmotivated singularities and are hence as geometric flows evidently not the best choice in the context of the calculus of variations, i.e. for energetic deformations.   
 
 However such gradient type flows, whether weak or strong, i.e. with respect to a $L^{2}$- or $W^{1,2}$-gradient,  are of interest in their own right apart from their usefulness in proving mere existence results to the underlying elliptic problem of prescribing the scalar curvature on a Riemannian manifold conformally, in particular due to the naturality  of $L^{2}$-gradient flows for a geometric problem. 
 
 \medskip 
 We wish to mention some works relevant to the flow analysis. 
 \begin{enumerate}[label=(\roman*)]
  \item  The most simple case evidently is, when the function $K$ to be prescribed is constant, e.g. $K=1$, and the underlying manifold is the standard sphere $\mathbb{S}^{n}$, in which case  flow convergence is known, cf. \cite{BahriCriticalPointsAtInfinity}, \cite{Ye}, with exponential speed, cf. \cite{short_proof_Brendle}. 
  \item  Later on and based on the positive mass theorem also on non spherical manifolds flow convergence in the Yamabe case $K=1$ was established, cf. \cite{Ye}, \cite{StruweLargeEnergies}, \cite{BrendleArbitraryEnergies}, with a subsequent analysis on upper and lower bounds of the speed of convergence, cf. \cite{SlowConvergence}. 
  \item  Returning to the spherical case $M=\mathbb{S}^{n}$, but considering a non constant function $K$ to be conformally prescribed as the scalar curvature, flows and their lack of compactness were first 
 analysed and characterised in  \cite{BahriCriticalPointsAtInfinity}, \cite{Bahri_Addendum_To_Bible} and \cite{Bahri_Coron_3_Sphere}
 in case $n=3$. For higher dimensional cases  we refer to \cite{BenAyed_n=4} for $n=4$ and  to \cite{MM3}
for $n\geq 5$, see also \cite{MM1},\cite{MM2} and \cite{MM4}.
  \item  Finally the case of a general Riemannian manifold $M$ with non constant $K$ to be prescribed, to which the present work belongs,  has been less studied with respect to an analysis of gradient flows. We point in case of a positive Yamabe invariant of $M$ to \cite{MM3} for a classification of non compactness in dimensions $n\geq 5$ and to \cite{may-cv} for some compactness results in dimensions $n=3,4,5$.  In case of a negative Yamabe invariant flow convergence was proven in \cite{Amacha_Regbaoui} recently.  
 \end{enumerate}

\noindent
In order to introduce the relevant notions,  consider a smooth, closed Riemannian manifold 
$$
M=(M^{n},g_{0}), \; n=3,4,5  
$$
with volume measure $\mu_{g_{0}}$ and scalar curvature $R_{g_{0}}$.
The Yamabe invariant 
\begin{equation*}\begin{split}
Y(M)
= &
\inf_{\mathcal{A}}
\frac
{\int c_{n}\vert \nabla u \vert_{g_{0}}^{2}+R_{g_{0}}u^{2}d\mu_{g_{0}}}
{(\int u^{\frac{2n}{n-2}}d\mu_{g_{0}})^{\frac{n-2}{n}}}
\; \text{ with }\;
c_{n}=4\frac{n-1}{n-2},
\end{split}\end{equation*}
where 
$$
\mathcal{A}=
\{
u\in W^{1,2}_{g_{0}}(M)\;:\; u\geq 0,u\not \equiv 0
\},
$$ 
is assumed to be positive. Then the  conformal Laplacian
$$
L_{g_{0}}=-c_{n}\Delta _{g_{0}}+R_{g_{0}}
$$
is  a positive,  selfadjoint operator with Green's function
$
G_{g_{0}}.
$ 
We may assume 
$$
R_{g_{0}}>0 
\; \text{ and } \;
\int Kd\mu_{g_{0}}=1
$$
for the background metric $g_{0}$.
For a conformal metric 
$$g=g_{u}=u^{\frac{4}{n-2}}g_{0}$$ 
there holds 
$
d\mu=d\mu_{g_{u}}=u^{\frac{2n}{n-2}}d\mu_{g_{0}}
$
for the volume element and 
\begin{equation*}\begin{split}
R=R_{g_{u}}=u^{-\frac{n+2}{n-2}}(-c_{n} \Delta _{g_{0}} u+R_{g_{0}}u)
=
u^{-\frac{n+2}{n-2}}L_{g_{0}}u
\end{split}\end{equation*} 
for the scalar curvature. We may define 
$$\Vert u \Vert^{2}=\int L_{g_{0}}uud\mu_{g_{0}}$$
and use $\Vert \cdot \Vert$ as an equivalent norm on $W^{1,2}(M)$.
Let $0<K\in C^{\infty}(M)$ and 
$$
r=r_{u}=\int Rd\mu, \,k=k_{u}=\int Kd\mu, \,
\K=\K_{u}
=
\frac{K}{k}.
$$
In \cite{may-cv} we have studied the $L^{2}$-pseudo gradient flow
\begin{equation}\begin{split}\label{Yamabe_flow}
\partial_{t}u
=
-(\frac{R}{K}-\frac{r}{k})u
\; \text{ on } \; 
X=\{u\in C^{\infty}(M,\R_{+}) \;:\; k =1 \}, 
\end{split}\end{equation}
which   evidently coincides with the Yamabe flow in case $K=1$. Obviously
$
\partial_{t}k=0,
$
i.e. 
the unit volume $k\equiv 1$ is preserved.  Let us consider the scaling invariant energy
\begin{equation}\label{functional}
\begin{split}
J(u)
= &
\frac{\int c_{n}\vert \nabla u \vert_{g_{0}}^{2}+R_{g_{0}}u^{2}d\mu_{g_{0}}}{(\int 
Ku^{\frac{2n}{n-2}}d\mu_{g_{0}})^{\frac{n-2}{n}}}
=\frac{\int L_{g_{0}}uud\mu_{g_{0}}}{(\int Ku^{\frac{2n}{n-2}}d\mu_{g_{0}})^{\frac{n-2}{n}}} 
\; \text{ for }\;  u\in \mathcal{A},
\end{split}
\end{equation} 
omitting from now on $d\mu_{g_{0}}$, when integrating with 
respect to it. 
\begin{proposition}\label{prop_derivatives_of_J}
We have 
$J(u)
=
\frac{r}{k^{\frac{n-2}{n}}}$ and
\begin{enumerate}[label=(\roman*)]
  \Item \quad  
\begin{equation*}\begin{split}
\frac{1}{2}\partial  J(u)v
= &
\frac{1}{k^{\frac{n-2}{n}}}
[
\int L_{g_{0}}uv
-
\frac
{
r
}
{
k
}
\int Ku^{\frac{n+2}{n-2}}v
]
= 
\frac{1}{k^{\frac{n-2}{n}}}\int (R-\frac{r}{k}K)u^{\frac{n+2}{n-2}}v
\end{split}\end{equation*}
 \Item \quad  
\begin{equation*}\begin{split}
\frac{1}{2}\partial^{2}  J(u)vw 
=&
\frac{1}{k^{\frac{n-2}{n}}}
[
\int L_{g_{0}}vw
-
\frac{n+2}{n-2}
\frac
{
r
}
{
k
} 
\int Ku^{\frac{4}{n-2}}vw
]
\\ &-
\frac{2}{k^{\frac{n-2}{n}+1}}
[
\int L_{g_{0}}uv\int Ku^{\frac{n+2}{n-2}}w
+
\int L_{g_{0}}uw\int Ku^{\frac{n+2}{n-2}}v
]\\
& +
4\frac{n-1}{n-2}\frac{r}{k^{\frac{n-2}{n}+2}}
\int Ku^{\frac{n+2}{n-2}}v\int Ku^{\frac{n+2}{n-2}}w.
\end{split}\end{equation*}
\end{enumerate}
Moreover $J$ is $C^{2, \alpha}_{loc}$ and uniformly H\"older continuous on each 
$$
U_{\varepsilon}=\{u\in \mathcal{A}\;:\;  \varepsilon<\Vert u \Vert,\,J(u)\leq \varepsilon^{-1}\}\subset \mathcal{A}.
$$
\end{proposition}
In particular  the problem of conformally  prescribing the scalar curvature is variational and 
\begin{equation*}\begin{split}  
\frac{1}{2}\vert \partial J(u)\vert
\leq \frac{1}{k^{\frac{n-2}{n}}}\Vert R-r\K \Vert_{L^{\frac{2n}{n+2}}_{\mu}}
\leq  \frac{1}{k^{\frac{n-2}{n}}}\Vert R-r\K \Vert_{L^{2}_{\mu}},
\end{split}\end{equation*}
where 
$
\vert \partial J(u)\vert
=
\vert \partial J(u)\vert_{W^{-1,2}_{g_{0}}(M)}. 
$
Then  by a slight abuse of notation we define 
$$
\vert \delta J\vert(u) =2k^{\frac{2-n}{n}}\Vert R-r\K\Vert_{L^{2}_{\mu}} 
$$
as a natural majorant of $\vert \partial J(u)\vert$ and along a flow line we have
\begin{equation*}
\begin{split}
\partial_{t} J(u)
\lesssim
-\vert \delta J(u)\vert^{2}
.
\end{split}
\end{equation*}
From Theorem 1 in \cite{may-cv} we know at least in cases $n=3,4,5$, that every  flow line for
\eqref{Yamabe_flow}
 exists positively for all times. Consequently  we have a priori  
$$
\int^{\infty}_{0}\vert \delta J(u)\vert^{2} dt<\infty,
$$
as  $J$ by positivity of the Yamabe invariant is lower bounded. Similarly  we may consider the  gradient flow
\begin{equation*} 
\partial_{t}u=-\nabla J(u),\; \nabla =\nabla^{L_{g_{0}}}, 
\end{equation*}
for which $\partial_{t}\Vert u \Vert=0$ instead of $\partial_{t}k=0$.  This describes a strong gradient flow, 
since  by definition
\begin{equation*}
\forall\; w \in W^{1,2}(M)\;:\;\langle \nabla J(u),w \rangle_{L_{g_{0}}}
=
\partial J(u)w
\; \text{ and }\; \Vert \cdot \Vert_{L_{g_{0}}}\simeq \Vert \cdot \Vert_{W^{1,2}},
\end{equation*}
and we write $\nabla J(u)=L_{g_{0}}^{-}\partial J(u)$. 
For the sake of easy comparability to \eqref{Yamabe_flow} consider
\begin{equation}\label{gradient_flow}
\partial_{t}u=-\frac{r}{2k}(\nabla J(u)-\frac{\int Ku^{\frac{n+2}{n-2}}\nabla J(u)}{k}u)
\end{equation}
as a strong pseudo gradient flow. Then $\partial_{t}k=0$ and, since   by scaling invariance we have 
$\partial J(u)u=0$,  there holds  under \eqref{gradient_flow} on $X$
\begin{equation*}
\partial_{t} J(u)=-\frac{r}{2k}\Vert  \nabla  J(u)\Vert^{2}=-\frac{r}{2}\vert \partial J(u)\vert^{2}.
\end{equation*}
In particular  and  by positivity of the Yamabe invariant we have along each flow line 
\begin{equation}\label{energy_bounds}
c(K)\leq J(u)=r_{u}=r=\int L_{g_{0}}uu=\Vert u \Vert^{2}\leq J(u_{0}). 
\end{equation}
Then, since 
\begin{equation}\label{gradient_positive}
\nabla J(u)
= 
L_{g_{0}}^{-}\partial J(u)
=
k^{\frac{2-n}{n}}  L_{g_{0}}^{-}(L_{g_{0}}u-r\bar Ku^{\frac{n+2}{n-2}})
\leq
\frac{u}{k^{\frac{n-2}{n}}}
\end{equation}   
by positivity of $L_{g_{0}}^{-}=G_{g_{0}}$, we find under \eqref{gradient_flow}
$$
\partial_{t} u \geq 
-
C(1+\vert \partial J(u)\vert)u, 
$$
so $u>0$ is preserved. Indeed due to  $k=1$ and \eqref{energy_bounds} we find from Proposition \ref{prop_derivatives_of_J}, that $\vert \partial J(u)\vert$ is a priori bounded along flow lines. Therefore  each flow line exists positively for all times and   
$$\partial_{t}J(u)\simeq -\vert \partial J(u)\vert^{2},$$
whence 
\begin{equation*}  
\int^{\infty}_{0}\Vert \nabla J(u)\Vert^{2}=\int^{\infty}_{0}\vert \partial  J(u)\vert^{2}<\infty.
\end{equation*}
We thus see, that \eqref{gradient_flow} defines a pseudo gradient flow on $X$ as well. Note, that \eqref{gradient_flow} falls into the class of ordinary differential equations, hence long time existence is a non issue in contrast to the  $L^{2}$- type flow \eqref{Yamabe_flow}.  
The difference, when considering \eqref{Yamabe_flow} in contrast to \eqref{gradient_flow} apart from the distinguishing quadratic a priori integrability of $\vert \delta J\vert$ versus $\vert \partial J\vert$ lies in the ease of adaptability. In fact considering a bounded and for instance smooth  vectorfield $W$ on $X$ satisfying 
$
\langle \nabla J,W\rangle \geq 0
$
we may modify \eqref{gradient_flow} to 
\begin{equation}\label{gradient_flow_modified}
\partial_{t}u=-\frac{r}{2k}(\nabla J(u)+W-\frac{\int Ku^{\frac{n+2}{n-2}}(\nabla J(u)+W)}{k}u),
\end{equation}
as we shall do in Section \ref{sec_modifying the gradient flow}. 
We then still decrease energy, find quadratic a priori integrability of $\vert \partial J\vert$,  preserve 
$\partial_{t}k=0$ and  $u>0$ and finally also \eqref{gradient_flow_modified} falls into the class of ordinary differential equations, hence also \eqref{gradient_flow_modified} defines  a flow on $X$. In contrast the long time existence of \eqref{Yamabe_flow} relies on higher order integrability properties of 
$R-r\bar K$, cf. \cite{BrendleArbitraryEnergies},\cite{may-cv}, which may be destroyed by even slight adaptations.

\

In any case, i.e. \eqref{Yamabe_flow},\eqref{gradient_flow} or \eqref{gradient_flow_modified}, the volume $k=1$  is preserved  and the lower bounded energy $J$ decreased, 
whence along a  flow line $u$ 
\begin{equation*} 
\int L_{g_{0}}uu=r=k^{\frac{n-2}{n}}J(u)=J(u)<J(u)\lfloor_{t=0}<\infty,
\end{equation*}
i.e. we have norm control along each flow line. Moreover  under \eqref{Yamabe_flow} there holds
\begin{equation}\label{flow_palais_smale_yamabe_flow}
\vert \partial J(u)\vert \lesssim \vert \delta J(u)\vert \longrightarrow 0 \; \text{ as }\; t \longrightarrow  \infty,
\end{equation} 
cf. Proposition  \ref{I-prop_strong_convergence_of_the_first_variation} in \cite{may-cv}. Likewise there holds under \eqref{gradient_flow_modified}
\begin{equation*}
\vert \partial J(u)\vert=\Vert \nabla J(u)\Vert \longrightarrow 0 \; \text{ as }\; t \longrightarrow \infty.
\end{equation*}
Indeed $\int^{\infty}_{0}\vert \partial J(u)\vert^{2}<\infty$ necessitates
\begin{equation*}
\vert \partial J(u_{t_{k}})\vert \longrightarrow 0
\; \text{ and }
\int^{\infty}_{t_{k}}\vert \partial J(u)\vert^{2} \longrightarrow 0
\; \text{ as }\; k \longrightarrow \infty
\end{equation*}
for a least a sequence $t_{k}\longrightarrow \infty$ as $k\longrightarrow \infty$ in time and thus for any $t>t_{k}$
\begin{equation}\label{flow_palais_smale_gradient} 
\vert \partial J(u_{t})\vert^{4}
\leq 
\vert \partial J(u_{t_{k}})\vert^{4}
+
C\int^{\infty}_{t_{k}} \vert \partial J(u)\vert^{2} \longrightarrow 0 \; \text{ as }\; k \longrightarrow \infty
\end{equation}
using a priori uniform boundedness of $\vert \partial J(u)\vert$ and $\vert \partial^{2}J(u)\vert$, cf. 
Proposition \ref{prop_derivatives_of_J}, along flow lines.

\   

Based on a fine description of a possible lack of compactness of flow lines, we had extracted suitable  assumptions to guarantee compactness of the flow on $X$ induced by \eqref{Yamabe_flow}, cf. Theorem 2 from \cite{may-cv}. For instance for $n=5$ under   
\textit{
\begin{enumerate}
\centering 
\item [\quad Cond$_{5}$:]\;\;\; $M$ is not conformally equivalent to the standard sphere $\mathbb{S}^{5}$ and
\begin{equation*}
 \langle \nabla \Delta  K, \nabla K\rangle 
>\frac{1}{3} \vert \Delta  K \vert^{2}
\end{equation*}
 on $\{\Delta  K<0\}\cap U$ for an open neighbourhood $U$ of 
$\{\nabla K=0\}$ 
\end{enumerate}   
}
\noindent 
every flow line for \eqref{Yamabe_flow} is compact and hence converges to a solution of $\partial J=0$ in $X$.  We will restrict our attention to the very simple scenario
\begin{Condition}\label{Condition_on_K}
Let $n=5$ and  
\begin{enumerate}[label=(\roman*)]
\item \quad $M\not \simeq \mathbb{S}^{5}$ conformally
 \item\quad $\exists \; x_{0}\in M\;:\; \{x_{0}\}=\{K=\max_{M}K\}$
 \item \quad$\Delta K >0$ on $\{x_{1},\ldots,x_{q}\}=\{\nabla K=0\} \setminus \{x_{0}\}$
 \item \quad in a conformal normal coordinate system around $x_{0}\simeq 0$ we have
$$
K(x)=1-\vert x \vert^{4},\; \text{ where }\;  \vert x \vert=(\sum_{i}x_{i}^{2})^{\frac{1}{2}}.
$$
\end{enumerate}
\end{Condition}   
\noindent
We refer to \cite{LeeAndParker} and \cite{Guenter} for the notion of conformal normal coordintates. 
Also note, that we  only slightly violate $Cond_{5}$, since indeed  close to $x_{0}$ we have 
\begin{equation*}
 \langle\nabla \Delta K, \nabla K\rangle=\frac{2}{n+2}\vert \Delta K\vert^{2}
 <\frac{1}{3}\vert \Delta K\vert^{2},
\end{equation*} 
in particular $Cond_{5}$ from \cite{may-cv} guaranteeing flow convergence is pretty sharp. 
As a consequence the only possible non compactness, i.e. non compact  flow lines for \eqref{Yamabe_flow} or \eqref{gradient_flow}, correspond to a bubbling close to $x_{0}$ with critical energy
\begin{equation*}
J_{\infty}=J(\varphi_{x_{0},\infty})
=\frac{c_{0}}{K^{\frac{n-2}{n}}(x_{0})}.
\end{equation*} 
This unique  bubbling then occurs both for \eqref{Yamabe_flow} and  \eqref{gradient_flow}
and we will compare these flows in detail. However by a slight modification of the latter flow in the spirit of \eqref{gradient_flow_modified} this non compactness will be completely removed. 
\begin{thm}
Let $M=(M^{n},g_{0})$ be a Riemannian manifold of dimension $n=5$ and positive Yamabe invariant.  
Then under  Condition \ref{Condition_on_K}   the flows generated by 
\begin{enumerate}[label=(\roman*)]
 \item \quad the Yamabe type, $L^{2}$-gradient flow \eqref{Yamabe_flow} and 
 \item \quad its normalised,  strong gradient type analogon \eqref{gradient_flow}
\end{enumerate}
for the prescribed scalar curvature functional \eqref{functional} exhibit exclusively non compact flow lines of single bubble type at the unique maximum of $K$, while there exists a compact pseudo gradient  for the latter functional, i.e. a pseudo gradient, all of whose flow lines are compact and hence converging. 
\end{thm}
\begin{proof}
We have seen above, that \eqref{Yamabe_flow} and \eqref{gradient_flow} induce a flow $\Phi$ on $X$, whose flow lines 
$$u=u_{t}=\Phi(t,u_{0})$$ up to a time sequence are Palais-Smale. 
Then up to a subsequence in time   
$$\forall\; \varepsilon>0 \;\exists\; N=N(\varepsilon)\in \N\;\forall\; n\geq N\;:\; u_{t_{n}}\in V(\omega,p,\varepsilon)$$
for 
\begin{enumerate}[label=(\roman*)]
 \item \quad  either $\omega=0$ and $p\in \N_{\geq 1}$
 \item \quad or  a solution $\omega>0$ to $\partial J(\omega)=0$ and $p\in \N_{\geq 0}$,  
\end{enumerate}
 c.f Definition \ref{def_V(omega,p,e)} and Proposition \ref{prop_concentration_compactness}.
 In fact $\omega=0$ and $p=0$ would imply 
 $$u_{t_{n}}\xrightarrow{n\to \infty}0\; \text{ strongly}$$  contradicting the normalisation $k=k_{u}=1$. 
 The latter statement is sharpened via Proposition \ref{prop_unicity_of_a_limiting_critical_point_at_infinity} to 
\begin{equation*}
\forall\; \varepsilon>0\; \exists\ T=T(\varepsilon)>0\; \forall t\geq T\;:\; u=u_{t}\in V(\omega,p,\varepsilon). 
\end{equation*}
Hence convergence in case $p=0$. By Section \ref{sec_compact_regions} only $p=1$ is possible  in case $p>0$ and then
$$a\xrightarrow{t\to \infty} x_{0}=\{K=\max K\}$$ for the single blow-up point $a$ of 
$$u=\alpha \varphi_{a,\lambda}+v\in V(p,\varepsilon)=V(0,p,\varepsilon).$$ 
Lemma \ref{lem_diverging} then shows, that indeed
$ \lambda\longrightarrow \infty$
for suitable initial data. Hence we have proven the exclusive existence of non compact flow lines as a single bubbling at $x_{0}$.

Finally for the modified flow on $X$ induced by \eqref{gradient_flow_modified_discussion}, which is a pseudo gradient flow by virtue of Lemma \ref{lem_modified_flow_energy_consumption}, the only possibility for a non compact flow line is as before a single bubbling scenario, cf. \eqref{concentration_compactness_for_modified_flow}, which is ruled out in Section \ref{sec_excluding_diverging_flow_lines}. Hence \eqref{gradient_flow_modified_discussion}
induces a compact flow. 
\end{proof}

  The plan of this work is as follows. In Section \ref{sec_preliminaries} we recall some preliminary notions already introduced in \cite{may-cv} for the study of such flows. In particular in Section \ref{sec_shadow_flow} we study the difference or rather the strict similarities of the shadow flow for \eqref{Yamabe_flow} and \eqref{gradient_flow}, i.e. the dynamics of those variables relevant to the underlying finite dimensional reduction. Subsequently we recall  in Section \ref{sec_principal_behaviour} some first and easy properties on flow lines based on this reduction. 
After this  lengthy exposition of introduction and preliminary results in Sections \ref{sec_introduction} and \ref{sec_preliminaries} we study in Section \ref{sec_divergence_and_compactification}  all possibilities of non compact flow lines for the flows induced by \eqref{Yamabe_flow} and \eqref{gradient_flow} and afterwards of a slight modification of the latter. 
Precisely we exclude in Section \ref{sec_compact_regions} all possibilities for non compact flow lines for \eqref{Yamabe_flow} and \eqref{gradient_flow}, which are not of single bubble type and concentrating at the maximum point of $K$. Subsequently in Section  \ref{sec_diverging_flow_lines} we show, that the latter remaining possibility is realised, i.e. that in fact such non compact flow lines exist for both flows. 
Finally we modify the latter flows in Section \ref{sec_modifying the gradient flow} and thus introducing a new  pseudo gradient flow, which in Section \ref{sec_excluding_diverging_flow_lines} is shown to be compact. 
Last and for the sake of readability we collect in the Appendix \ref{sec_appendix}  some statements from \cite{may-cv} and a proof from Section \ref{sec_preliminaries}.

\section{Preliminaries}\label{sec_preliminaries}  

As we had seen via \eqref{flow_palais_smale_yamabe_flow} and \eqref{flow_palais_smale_gradient}, every flow line for \eqref{Yamabe_flow} and \eqref{gradient_flow_modified}  up to the choice of a time sequence constitutes a Palais-Smale sequence for $J$, whose   possible lack of compactness we now describe.   

\begin{definition}
 
For $a\in M$ let $u_{a}$ via $g_{a}=u_{a}^{\frac{4}{n-2}}g_{0}$ introduce conformal  normal coordinates and let  $G_{g_{ a }}$ be the Green's function of the conformal Laplacian $L_{g_{a}}$.
For $\lambda>0$ let
\begin{equation*}\begin{split}
\varphi_{a, \lambda }
= &
u_{ a }(\frac{\lambda}{1+\lambda^{2} \gamma_{n}G^{\frac{2}{2-n}}_{ a }})^{\frac{n-2}{2}},
\;G_{ a }=G_{g_{ a }}( a, \cdot), \;
\gamma_{n}=(4n(n-1)\omega _{n})^{\frac{2}{2-n}}.
\end{split}\end{equation*}
One may expand 
$
G_{ a }=\frac{1}{4n(n-1)\omega _{n}}(r^{2-n}_{a}+H_{ a })
$
with
$
r_{a}=d_{g_{a}}(a, \cdot)
$
and decompose
\begin{equation*}
\begin{split}
H_{ a }=H_{r,a }+H_{s, a },\;H_{r,a }\in C^{2, \alpha}_{loc},\;
H_{s,a}
=
O
\begin{pmatrix}
0 & \text{ for }\, n=3\\ r_{a}^{2}\ln r_{a} & \text{ for }\, n=4 \\ r_{a}& \text{ for }\, n=5 
\end{pmatrix}.
\end{split}
\end{equation*} 
In addition 
the positive  mass theorem tells, that $H_{a}(a)\geq 0 $ for all $a\in M$ and 
\begin{equation*}
H_{a}(a)=0 
\;\text{ for }\; 
M\simeq \mathbb{S}^n,
\text{ while }\; 
H_{a}(a)>0
\; \text{ for }\; 
M \not \simeq \mathbb{S}^n
\end{equation*}
in the sense of  conformal equivalence. 
\end{definition}  

We abbreviate some notation. 

\begin{definition}\label{def_relevant_quantities}
For $k,l=1,2,3$ and $ \lambda_{i} >0, \, a _{i}\in M, \,i= 1, \ldots,p$ define
\begin{enumerate}[label=(\roman*)]
  \item \quad  
$\var_{i}=\var_{a_{i}, \lambda_{i}}$ and $(d_{1,i},d_{2,i},d_{3,i})=(1,-\lambda_{i}\partial_{\lambda_{i}}, \frac{1}{\lambda_{i}}\nabla_{a_{i}})$
  \item \quad 
$\phi_{1,i}=\varphi_{i}, \;\phi_{2,i}=-\lambda_{i} \partial_{\lambda_{i}}\varphi_{i}, \;\phi_{3,i}= \frac{1}{\lambda_{i}} \nabla_{ a _{i}}\varphi_{i}$, so
$
\phi_{k,i}=d_{k,i}\varphi_{i}
$
 \item \quad  
$K_{i}=K(a_{i}),\nabla K_{i}=\nabla K(a_{i})$ and so on. 
\end{enumerate}
\end{definition}

Let us collect some standard interaction estimates for these bubbles. 

\begin{lemma}\label{lem_interactions} 
Let $k,l=1,2,3$ and $i,j = 1, \ldots,p$. We have
\begin{enumerate}[label=(\roman*)]
  \item \quad 
$ \vert \phi_{k,i}\vert, \vert \lambda_{i}\partial_{\lambda_{i}}\phi_{k,i}\vert, \vert \frac{1}{\lambda_{i}}\nabla_{a_{i}} \phi_{k,i}\vert\leq C \varphi_{i}$
  \item \quad 
$ 
\int \varphi_{i}^{\frac{4}{n-2}} \phi_{k,i}\phi_{k,i}
=
c_{k}\cdot id
+
O(\lambda_{i}^{2-n}+\lambda_{i}^{-2}), \;c_{k}>0$
 \item \quad 
$
\int \varphi_{i}^{\frac{n+2}{n-2}}\phi_{k,j}
= 
b_{k}d_{k,i}\varepsilon_{i,j}+o_{\varepsilon}(\varepsilon_{i,j})
=
\frac{n+2}{n-2}\int \phi_{k,i}\varphi_{i}^{\frac{4}{n-2}}\varphi_{j}, \; b_{k}>0,
\,i\neq j$
  \item \quad  
$
\int \varphi_{i}^{\frac{4}{n-2}} \phi_{k,i}\phi_{l,i}
= 
O( \lambda_{i}^{2-n} +\lambda_{i}^{-2})$
for $k\neq l$, 
$\int \varphi_{i}^{\frac{2n}{n-2}}=c_{1}+O(\lambda_{i}^{2-n})$ and
$$
\int \varphi_{i}^{\frac{n+2}{n-2}} \phi_{k,i}
= 
O( \lambda_{i}^{2-n})\; \text{ for }\; k=2,3
$$
  \item \quad 
$
\int \varphi_{i}^{\alpha}\varphi_{j}^{\beta} 
=
O(\varepsilon_{i,j}^{\beta})
$
for $i\neq j$ and $\alpha +\beta=\frac{2n}{n-2}, 
\;\alpha>\frac{n}{n-2}>\beta\geq 1 $
 \item \quad 
$
\int \varphi_{i}^{\frac{n}{n-2}}\varphi_{j}^{\frac{n}{n-2}} 
=
O(\varepsilon^{\frac{n}{n-2}}_{i,j}\ln \varepsilon_{i,j}), \,i\neq j
$
  \item \quad  
$
(1, \lambda_{i}\partial_{\lambda_{i}}, \frac{1}{\lambda_{i}}\nabla_{a_{i}})\varepsilon_{i,j}=O(\varepsilon_{i,j})
, \,i\neq j$,
\end{enumerate}
where 
\begin{enumerate}
 \item[$1.)$] \quad
$\varepsilon=\min\{\frac{1}{\lambda_{i}}, \frac{1}{\lambda_{j}}, 
\varepsilon_{i,j}\}
,\;
\varepsilon_{i,j}
=
(
\frac{\lambda_{i}}{\lambda_{j}}
+ 
\frac{\lambda_{j}}{\lambda_{i}}
+
\lambda_{i}\lambda_{j}\gamma_{n}G_{g_{0}}^{\frac{2}{2-n}}(a _{i},a _{j})
)^{\frac{2-n}{2}}
$
 \item[$2.)$] \quad 
$
c_{1}
=
\int_{\R^{n}}\frac{1}{(1+r^{2})^{n}}
,\;
c_{2}
= 
\frac{(n-2)^{2}}{4}\int_{\R^{n}} \frac{\vert r^{2}-1\vert^{2}}{(1+r^{2})^{n+2}}
,\;
c_{3}
=
\frac{(n-2)^{2}}{n}\int_{\R^{n}} \frac{r^{2}}{(1+r^{2})^{n+2}}.
$
\end{enumerate}
\end{lemma}
\begin{proof}
 Cf. \ref{I-lem_interactions} in \cite{may-cv}.
\end{proof}
For a better description of the gradient  we decompose the  second variation. 
To that end we recall from \cite{may-cv}, 
cf. 
Lemma \ref{I-lem_degeneracy_and_pseudo_critical_points} and 
Proposition \ref{I-prop_smoothness_of_u_a_b},
\begin{lemma} 
For $\omega > 0$ solving 
$$L_{g_{0}}\omega =K\omega  ^{\frac{n+2}{n-2}}$$
there exist $\varepsilon>0$, an open neighbourhood $U$ of $\omega$ and 
$$
h:B_{\varepsilon}^{\R^{m+1}}(0)\longrightarrow H_{0}(\omega )^{\perp_{L_{g_{0}}}}
,\; 
H_{0}(\omega)=ker \partial^{2}J(\omega)
$$
smooth such, that
\begin{equation*}
\begin{split}
\{
w\in U 
& 
\;:\; \Pi_{H_{0}(\omega )^{\perp_{L_{g_{0}}}}} \nabla J(w)=0
\}
\\
= &
\{
u_{\alpha, \beta}
=
(1+\alpha) \omega +\beta^{i}\mathrm{e}_{i}
+
h(\alpha, \beta)\;:\; 
(\alpha, \beta)\in B_{\varepsilon}^{m+1}(0)\},
\end{split}
\end{equation*} 
where  $\{\omega,\mathrm{e}_{i}\;:\; i=1,\ldots,m\}\in ONB_{L_{g_{0}}}(ker \partial^{2}J(\omega))$ and 
$$\Vert h(\alpha, \beta)\Vert =O(\vert \alpha\vert^{2}+\vert \beta \vert^{2}).$$
We call $w \in U$ a pseudo critical point related to $\omega$, if 
$$\Pi_{H_{0}(\omega)^{\perp_{L_{g_{0}}}}} \nabla J(w)=0.$$
Moreover there holds
$
\vert h(\alpha, \beta)\vert_{C^{k}}\longrightarrow 0\; \text{ as }\;  \vert \alpha\vert+
\vert \beta \vert \longrightarrow 0
$
for any $k\in \N$.
\end{lemma} 

 We may thereby define a neighbourhood of, where a loss of compactness, if present, has to occur. 

\begin{definition}\label{def_V(omega,p,e)} 
Let  
$\omega\geq 0$ solve $L_{g_{0}}\omega =K\omega  ^{\frac{n+2}{n-2}}, \,p\in \N$ and $\varepsilon>0$. 
Let for $u\in X$
\begin{equation*}
\begin{split}
A_{u}(\omega,p, \varepsilon)
=
\{ &
(\alpha, \beta_{k}, \alpha_{i}, \lambda_{i},a_{i})\in (\R_{+}, \R^{m}, \R^{p}_{+}, \R^{p}_{+},M^{p}) \;:\;
\\
& \;
\underset{i\neq j}{\fa}\;
 \lambda_{i}^{-1}, \lambda_{j}^{-1}, \varepsilon_{i,j}, \vert 1-\frac{r\alpha_{i}^{\frac{4}{n-2}}K(a_{i})}{4n(n-1)k}\vert,
\\
& \quad\;\;\;
\vert 1-\frac{r\alpha^{\frac{4}{n-2}}}{k}\vert, \vert \beta \vert,
\Vert u-u_{\alpha, \beta}-\alpha^{i}\varphi_{a_{i}, \lambda_{i}}\Vert
<\varepsilon\
\}.
\end{split}
\end{equation*} 
We define 
\begin{equation*}\begin{split}
V(\omega, p, \varepsilon)
= 
\{
u\in X  
\;:\;
A_{u}(\omega,p, \varepsilon)\neq \emptyset
\}
\end{split}\end{equation*}
and call $V(\omega,p,\varepsilon)$ in case $p>0$ a neighbourhood of a potential critical point at infinity.
\end{definition}
Note, that
$u_{\alpha, \beta}=0$, if $\omega=0$, and the conditions on $\alpha$ and $\beta_{k}$ become trivial. 
Moreover  either
$w=0$ or $w>0$
due to the strong maximum principle.

\begin{proposition}\label{prop_concentration_compactness} 
Every Palais-Smale sequence of $J$ in $X$
is precompact in some $V(\omega,p,\varepsilon)$, i.e.
\begin{equation*}
\forall \; t_{k}\longrightarrow \infty \; \exists\; (t_{k_{l}})\subset (t_{k})\;:\; 
u_{t_{k_{l}}}\in V(\omega,p,\varepsilon),
\end{equation*}
for every $\varepsilon>0$.
\end{proposition}
This characterisation of  lack of compactness is classical like the subsequent 
reduction by minimisation and we refer to 
\cite{BahriCoronCriticalExponent},\cite{may-cv} and \cite{StruweConcentrationCompactness}.  

\begin{proposition}\label{prop_optimal_choice} 
For every $\varepsilon_{0}>0$  there exists $\varepsilon_{1}>0$ such, that  for 
$$u\in V(\omega, p, \varepsilon)\; \text{ with } \;\varepsilon<\varepsilon_{1}$$
the minimisation problems
\begin{enumerate}[label=(\roman*)]
\item \quad
$
\inf
_
{
(\tilde \alpha, \tilde\beta_{k}, \tilde\alpha_{i}, \tilde a_{i}, \tilde\lambda_{i})\in A_{u}(\omega,p,2\varepsilon_{0}) 
}
\int 
Ku^{\frac{4}{n-2}}
\vert 
u
-
u_{\tilde \alpha, \tilde \beta}
-
\tilde\alpha^{i}\varphi_{\tilde a_{i}, \tilde \lambda_{i}}
\vert^{2}
$
\item \quad 
$
\inf
_
{
(\tilde \alpha, \tilde\beta_{k}, \tilde\alpha_{i}, \tilde a_{i}, \tilde\lambda_{i})\in A_{u}(\omega,p,2\varepsilon_{0}) 
}
\Vert 
u
-
u_{\tilde \alpha, \tilde \beta}
-
\tilde\alpha^{i}\varphi_{\tilde a_{i}, \tilde \lambda_{i}}
\Vert^{2}
$
\end{enumerate}
admit each a unique minimise $(\alpha, \beta_{k}, \alpha_{i},a_{i}, \lambda_{i})\in A_{u}(\omega,p, \varepsilon_{0})$ and we define
\begin{equation*}\begin{split}
\varphi_{i}=\varphi_{a_{i}, \lambda_{i}},v=u-u_{\alpha, \beta}-\alpha^{i}\varphi_{i},
\;\, \varepsilon_{i,j}
=
(
\frac{\lambda_{j}}{\lambda_{i}}
+
\frac{\lambda_{i}}{\lambda_{j}}
+
\lambda_{i}\lambda_{j}\gamma_{n}G_{g_{0}}^{\frac{2}{2-n}}(  a _{i},  a _{j})
)^{\frac{2-n}{2}}
\end{split}\end{equation*}
depending on the chosen minimisation.    
Moreover 
$$(\alpha, \beta_{k}, \alpha_{i},a_{i}, \lambda_{i}) \; \text{ and }\; v$$ 
depend smoothly on $u$.
\end{proposition} 
\noindent
The above minimisations evidently induce orthogonal properties for  
$$v=u-u_{\alpha, \beta}-\alpha^{i}\var_{i}$$    
with respect to the scalar products
\begin{equation*}
\langle a,b\rangle_{Ku^{\frac{4}{n-2}}}
=
\int Ku^{\frac{4}{n-2}}ab
\; 
\text{ or }
\; 
\langle a,b\rangle_{L_{g_{0}}}
=
\int L_{g_{0}}ab
\end{equation*}
respectively. This justifies to define the  orthogonal spaces, on which $v$ lives.   

\begin{definition}
 
For $u\in V(\omega, p, \varepsilon)$ let
\begin{equation*}\begin{split}
H_{u}(\omega, p, \varepsilon)
=
\langle 
u_{\alpha, \beta}, \partial_{\beta_{i}}u_{\alpha, \beta}, \varphi_{i},-\lambda_{i}\partial_{\lambda_{i}}\varphi_{i}, \frac{1}{\lambda_{i}}\nabla_{a_{i}}\varphi_{i}
\rangle
^{\perp_{Ku^{\frac{4}{n-2}}}}
\end{split}\end{equation*}
or  respectively 
\begin{equation*}\begin{split}
H_{u}(\omega, p, \varepsilon)
=
\langle 
u_{\alpha, \beta}, \partial_{\beta_{i}}u_{\alpha, \beta}, \varphi_{i},-\lambda_{i}\partial_{\lambda_{i}}\varphi_{i}, \frac{1}{\lambda_{i}}\nabla_{a_{i}}\varphi_{i}
\rangle
^{\perp_{L_{g_{0}}}}
\end{split}\end{equation*}
in case $\omega>0$. In case $\omega=0$ let $H_{u}(0,p,\varepsilon)=H_{u}(p,\varepsilon)$ and 
\begin{equation*}\begin{split}
H_{u}(p, \varepsilon)
=
\langle 
\varphi_{i},-\lambda_{i}\partial_{\lambda_{i}}\varphi_{i}, \frac{1}{\lambda_{i}}\nabla_{a_{i}}\varphi_{i}
\rangle
^{\perp_{Ku^{\frac{4}{n-2}}}}
\end{split}\end{equation*}
or respectively 
\begin{equation*}\begin{split}
H_{u}(p, \varepsilon)
=
\langle 
\varphi_{i},-\lambda_{i}\partial_{\lambda_{i}}\varphi_{i}, \frac{1}{\lambda_{i}}\nabla_{a_{i}}\varphi_{i}
\rangle
^{\perp_{L_{g_{0}}}}.
\end{split}\end{equation*}
\end{definition}
Recalling Definition \ref{def_relevant_quantities} and $u_{\alpha,\beta}=0$ in case $\omega=0$ we may simply write 
$$
H_{u}(\omega,p,\varepsilon)=\langle u_{\alpha,\beta},\partial_{\beta_{i}}u_{\alpha,\beta}, \phi_{k,i}\rangle^{\perp_{\cdot}},
\; \text{ in particular} \;   
H_{u}(p,\varepsilon)=\langle   \phi_{k,i}\rangle^{\perp_{\cdot}}
$$  
depending on the chosen minimisation.  These orthogonalities differ only a little,
as the next Lemma, whose proof we delay to Appendix \ref{sec_appendix}, quantifies.  
\begin{lemma}\label{lem_comparing_orthogonalities}
Let $\nu_{1}\in H_{u}(\omega, p, \varepsilon)=\langle u_{\alpha,\beta},\partial_{\beta_{i}}u_{\alpha,\beta},\phi_{k,i}\rangle
^{\perp_{Ku^{\frac{4}{n-2}}}}$. Then 
\begin{enumerate}[label=(\roman*)]
 \item \quad 
 $
 \Pi^{\top_{L_{g_{0}}}}_{\langle \phi_{k,i}\rangle}
\nu_{1}
=
O
(
(
\frac{\vert  \nabla K_{i}\vert}{\lambda_{i}}
+
\frac{1}{\lambda_{i}^{2}}
+
\frac{1}{\lambda_{i}^{n-2}}
+
\sum_{j\neq i}\varepsilon_{i,j}
+
\Vert v \Vert)
\Vert \nu_{1} \Vert)
$
for $\omega=0$
\item \quad   
 $
 \Pi^{\top_{L_{g_{0}}}}_{\langle u_{\alpha,\beta},\partial_{\beta_{i}}u_{\alpha,\beta} ,\phi_{k,i}\rangle}
\nu_{1}
=
O
(
(
\frac{\vert  \nabla K_{i}\vert}{\lambda_{i}}
+
\frac{1}{\lambda_{i}^{\frac{n-2}{2}}}
+
\sum_{j\neq i}\varepsilon_{i,j}
+
\Vert v \Vert)
\Vert \nu_{1} \Vert)
$ 
for $\omega>0$.
\end{enumerate} 
Conversely for $\nu_{2}\in H_{u}(\omega, p, \varepsilon)=\langle u_{\alpha,\beta}, \partial_{\beta_{i}}u_{\alpha,\beta},\phi_{k,i}\rangle
^{\perp_{L_{g_{0}}}}$ there holds 
\begin{enumerate}[label=(\roman*)]
 \item \quad 
 $\Pi^{\top_{Ku^{\frac{4}{n-2}}}}_{\langle \phi_{k,i}\rangle}
\nu_{2}
=
O
(
(
\frac{\vert  \nabla K_{i}\vert}{\lambda_{i}}
+
\frac{1}{\lambda_{i}^{2}}
+
\frac{1}{\lambda_{i}^{n-2}}
+
\sum_{j\neq i}\varepsilon_{i,j}
+
\Vert v \Vert)\Vert \nu_{2} \Vert)$
for  $\omega=0$
 \item \quad 
  $\Pi^{\top_{Ku^{\frac{4}{n-2}}}}_{\langle u_{\alpha,\beta},\partial_{\beta_{i}}u_{\alpha,\beta},\phi_{k,i}\rangle}
\nu_{2}
=   
O
(
(
\frac{\vert  \nabla K_{i}\vert}{\lambda_{i}}
+
\frac{1}{\lambda_{i}^{\frac{n-2}{2}}}
+
\sum_{j\neq i}\varepsilon_{i,j}
+
\Vert v \Vert)
\Vert \nu_{2} \Vert)
$
for $\omega>0$.
 \end{enumerate}
\end{lemma} 
The aforegoing Lemma will help us to carry over several estimates from \cite{may-cv}, which was based 
on a representation $u=\alpha^{i}\varphi_{i}+v$ with orthogonalities
\begin{equation*}
\langle \phi_{k,i},v\rangle_{Ku^{\frac{4}{n-2}}}=0
\end{equation*}
from  the first minimisation problem in Proposition \ref{prop_optimal_choice}.

\begin{proposition}\label{prop_positivity_of_D2J}
There exist $\gamma, \varepsilon_{0}>0$ such,  that for any $0<\varepsilon<\varepsilon_{0}$ and
$$
u=\alpha^{i}\varphi_{i}+v\in V(p,\varepsilon) 
$$
there holds 
$
\partial^{2}J(\alpha^{i}\varphi_{i})\lfloor_{H}>\gamma
$ 
for 
$H=H_{u}(p, \varepsilon).
$
\end{proposition} 
This positivity property is well known in either case
\begin{equation*}
H_{u}(p, \varepsilon)
=
\langle 
\phi_{k,i}
\rangle
^{\perp_{Ku^{\frac{4}{n-2}}}}
\; \text{ or }\;
H_{u}(p, \varepsilon)
=
\langle 
\phi_{k,i}
\rangle
^{\perp_{L_{g_{0}}}}
\end{equation*}
and evidently one case follows from the other by virtue of Lemma \ref{lem_comparing_orthogonalities}.
Likewise in case $u\in V(\omega,p,\varepsilon)$, cf.  Proposition \ref{I-prop_decomposing the second variation_f} from  \cite{may-cv}. 


\begin{proposition}\label{prop_decomposing the second variation_f}
There exist $\gamma, \varepsilon_{0}>0$ such,  that for any 
$$
u=u_{\alpha,\beta}+\alpha^{i}\varphi_{i}+v\in V(\omega,p,\varepsilon) 
$$
with $0<\varepsilon<\varepsilon_{0}$ we may decompose 
\begin{equation*}\begin{split}
H_{u}(\omega, p, \varepsilon)=H=H_{+}\oplus_{L_{g_{0}}} H_{-}\;\text{ with }\;\dim H_{-}<\infty
\end{split}\end{equation*}
and for any $h_{+}\in H_{+},h_{-}\in H_{-}$ there holds
\begin{enumerate}[label=(\roman*)]
 \item \quad 
$\partial^{2}J(u_{\alpha, \beta}+\alpha^{i}\varphi_{i})\lfloor_{H_{+}}>\gamma$ 
and
$\partial^{2}J(u_{\alpha, \beta}+\alpha^{i}\varphi_{i})\lfloor_{H_{-}}<-\gamma$
 \item \quad 
$\partial^{2}J(u_{\alpha, \beta}+ \alpha^{i}\varphi_{i})h_{+}h_{-}=o_{\varepsilon}(\Vert h_{+} \Vert \Vert h_{-}\Vert)$.
 \end{enumerate}
\end{proposition} 

The invertibility of the second variation on the orthogonal space, on which $v$ lives, then provides a priori estimates.

\begin{proposition}\label{prop_a-priori_estimate_on_v} 
For $\varepsilon>0$ small we have 
\begin{enumerate}[label=(\roman*)]
 \item \quad 
 $\Vert v \Vert
=
O(\sum_{r} \frac{\vert \nabla K_{r}\vert}{\lambda_{r}} + \frac{\vert \lap K_{r}\vert}{\lambda_{r}^{2}}
+
\lambda_{r}^{2-n} +\sum_{r\neq s}\varepsilon_{r,s}
+
\vert \partial J(u)\vert)$ on $V(p, \varepsilon)$
\item \quad 
$\Vert v \Vert
= 
O
(
\sum_{r} \frac{\vert \nabla K_{r}\vert}{\lambda_{r}}
+
\lambda_{r}^{\frac{2-n}{2}} 
+
\sum_{r\neq s}\varepsilon_{r,s}
+
\vert \partial  J(u)\vert)$
on $V(\omega, p, \varepsilon)$  
\end{enumerate}
\end{proposition} 
\begin{proof}
The statement for $V(p,\varepsilon)$  follows by expanding
$$
\partial J(u)v=\partial J(\alpha^{i}\varphi_{i}+v)v
$$
in $v$ and applying  Propositions  \ref{prop_positivity_of_D2J} and  \ref{prop_derivatives_on_H}.
Likewise the statement for  $V(\omega,p,\varepsilon)$ follows by expanding 
$$
\partial J(u)v_{\pm}=\partial J(u_{\alpha,\beta}+\alpha^{i}\varphi_{i}+v)v_{\pm}  
$$
in $v$ and applying Proposition  \ref{prop_derivatives_on_H_f} and \ref{prop_decomposing the second variation_f}, where we denote by
\begin{equation*}
v_{+}=\Pi^{\top_{L_{g_{0}}}}_{H_{+}}v
\; \text{ and }\; 
v_{-}=\Pi^{\top_{L_{g_{0}}}}_{H_{-}}v
\end{equation*}
the corresponding projections onto $H_{+}$ and $H_{-}$ in Proposition \ref{prop_decomposing the second variation_f}. 
\end{proof}

These estimates on $v$ are upon the appearance of $\vert  \partial J(u)\vert $ instead of $\vert \delta J(u)\vert$ the same as in \cite{may-cv}, cf. Corollaries \ref{I-cor_a-priori_estimate_on_v} and \ref{I-cor_a-priori_estimate_on_v_f} therein.  In fact in the latter work we had too graciously estimated against $\vert \delta J(u)\vert$ in many cases. In what follows we will simply give the correct statements without repeating the various proofs from \cite{may-cv}.

\subsection{The shadow flows}\label{sec_shadow_flow}
We recall some standard testings of the first variation
$$
\partial J(u)
=
\frac{2}{k^{\frac{n-2}{n}}}
[
\int L_{g_{0}}uv
-
\frac
{
r
}
{
k
}
\int Ku^{\frac{n+2}{n-2}}v
],
$$
cf. Proposition \ref{prop_derivatives_of_J}.  
\begin{proposition}\label{prop_simplifying_ski} 
For $u\in V(\omega,p, \varepsilon)$   and $\varepsilon>0$ sufficiently small let 
\begin{equation*}\begin{split}
\sigma_{k,i}=-\int (L_{g_{0}}u-r\K u^{\frac{n+2}{n-2}})\phi_{k,i}, \, i=1, \ldots,p, \,k=1,2,3.
\end{split}\end{equation*} 
Then in case $\omega=0$ we have  
with constants $b_{2}, \ldots,e_{4}>0$ 
\begin{enumerate}[label=(\roman*)]
 \item \quad 
$
\sigma_{2,i}
= 
d_{2}\alpha_{i}\frac{ H_{i}}{\lambda_{i} ^{n-2}}
+
e_{2}\frac{r\alpha_{i}^{\frac{n+2}{n-2}}}{k}\frac{\lap K_{i}}{\lambda_{i}^{2}} 
- 
b_{2}\frac{r\alpha_{i}^{\frac{4}{n-2}}K_{i}}{k}
\sum_{i \neq j=1}^{p}\alpha_{j}
\lambda_{i}\partial_{\lambda_{i}}\varepsilon_{i,j} 
$
 \item  \quad 
$      
\sigma_{3,i}
= 
\frac{r\alpha_{i}^{\frac{n+2}{n-2}}}{k}
[
 e_{3}\frac{\nabla K_{i}}{\lambda_{i}}
+
e_{4}\frac{\nabla \lap K_{i}}{\lambda_{i}^{3}}
] 
+
b_{3}\frac{r\alpha_{i}^{\frac{4}{n-2}}K_{i}}{k}
\sum_{i \neq j=1}^{p}
\frac{\alpha_{j}}{\lambda_{i}}\nabla_{a_{i}}\varepsilon_{i,j} 
$
\end{enumerate}
up to some 
$$
o_{\varepsilon}
(
 \lambda_{i}^{2-n} 
+
\sum_{i\neq j=1}^{p}\varepsilon_{i,j}
)
+
O
(
\sum_{r\neq s}\varepsilon_{r,s}^{2} 
+
\Vert v \Vert^{2}
+
\vert \partial  J(u)\vert^{2}
)  
,                              
$$
whereas  in case $\omega>0$ with  constants $d_{2},\ldots,b_{3}>0$
\begin{enumerate}[label=(\roman*)]
 \item \quad 
$
\sigma_{2,i}
= 
d_{2}\frac{r\alpha_{i}^{\frac{4}{n-2}}}{k}\frac{\alpha \omega_{i}}{\lambda_{i} ^{\frac{n-2}{2}}}
-b_{2} \frac{r\alpha_{i}^{\frac{4}{n-2}}K_{i}}{k}
\sum_{i\neq j =1}^{p}\alpha_{j}
\lambda_{i}\partial_{\lambda_{i}}\varepsilon_{i,j}
$
 \item \quad  
$
\sigma_{3,i}
= 
d_{3}\frac{r\alpha_{i}^{\frac{n+2}{n-2}}}{k}\frac{\nabla K_{i}}{\lambda_{i}}
+ b_{3}\frac{r\alpha_{i}^{\frac{4}{n-2}}K_{i}}{k}
\sum_{i\neq j =1}^{p}\alpha_{j}
\frac{1}{\lambda_{i}}\nabla_{a_{i}}\varepsilon_{i,j}
$
\end{enumerate}
 up to some 
$$
o_{\varepsilon}
(
\lambda_{i}^{\frac{2-n}{2}}+\sum_{i\neq j=1}^{p}\varepsilon_{i,j}
)
+
O
(
\sum_{r\neq s}\varepsilon_{r,s}^{2}
+
\Vert v \Vert^{2}
+ 
\vert \partial J(u)\vert^{2}
).
$$
\end{proposition} 
\begin{proof}
Cf. Corollaries \ref{I-cor_simplifying_ski} and \ref{I-cor_simplifying_ski_f} in \cite{may-cv}.
\end{proof}
So far and in contrast to \cite{may-cv} we have removed the appearance of $\vert \delta J\vert$. In fact only in the computation of the shadow flow, i.e. the description of the movements of $\alpha_{i},\lambda_{i}$ and $a_{i}$ this error term inevitably enters.   
  
\begin{proposition} \label{prop_the_shadow_flow}
For $u\in V(p, \varepsilon)$ with $\varepsilon>0$ small we have
\begin{enumerate}[label=(\roman*)] 
 \item \quad 
$
-\frac{\dot \lambda_{i}}{\lambda_{i}}
= 
\frac{r}{k}
[
\frac{d_{2}}{c_{2}}
\frac{ H_{i}}{\lambda_{i} ^{n-2}}
+
\frac{e_{2}}{c_{2}}\frac{\lap K_{i}}{K_{i}\lambda_{i} ^{2}} 
-
\frac{b_{2}}{c_{2}} \sum_{i \neq j=1}^{p}\frac{\alpha_{j}}{\alpha_{i}}
\lambda_{i}\partial_{\lambda_{i}}\varepsilon_{i,j} 
]
(1+o_{\frac{1}{\lambda_{i}}}(1)) 
$
 \item \quad 
$
\lambda_{i}\dot a_{i}
= 
\frac{r}{k}
[
\frac{e_{3}}{c_{3}}  \frac{\nabla K_{i}}{K_{i}\lambda_{i}}
+
\frac{e_{4}}{c_{3}}  \frac{\nabla \lap K_{i}}{K_{i}\lambda_{i}^{3}}
+
\frac{b_{3}}{c_{3}} \sum_{i \neq j=1}^{p}\frac{\alpha_{j}}{\alpha_{i}}
 \frac{1}{\lambda_{i}}\nabla_{a_{i}}\varepsilon_{i,j}
](1+o_{\frac{1}{\lambda_{i}}}(1)) 
$
\end{enumerate}
up to some $o_{\varepsilon}  ( \lambda_{i}^{2-n}  +\sum_{i\neq j=1}^{p}\varepsilon_{i,j})$ and 
\begin{equation*} 
O
(
\sum_{r\neq s} \frac{\vert \nabla K_{r}\vert^{2}}{\lambda_{r} ^{2}}
+
\frac{\vert \lap K_{r}\vert^{2}}{\lambda_{r}^{4}}
+
\lambda_{r}^{-2(n-2)}
+
\varepsilon_{r,s}^{2}
)
+ 
\begin{cases}
O
(
\vert \delta J(u)\vert^{2}
)
\; \text{ under } \eqref{Yamabe_flow} 
\\

O
(
\vert \partial  J(u)\vert^{2}
)
\; \text{ under } \eqref{gradient_flow} 
\end{cases}.
\end{equation*}
For $u\in V(\omega, p, \varepsilon)$ with $\varepsilon>0$ small we have 
\begin{enumerate}[label=(\roman*)]
 \item \quad 
$
-\frac{\dot \lambda_{i}}{\lambda_{i}}
= 
\frac{r}{k}
[
\frac{d_{2}}{c_{2}}
\frac{\alpha \omega_{i}}{\alpha_{i}K_{i}\lambda_{i} ^{\frac{n-2}{2}}}
-
\frac{b_{2}}{c_{2}}\sum_{i\neq j =1}^{p}\frac{\alpha_{j}}{\alpha_{i}}
\lambda_{i}\partial_{\lambda_{i}}\varepsilon_{i,j}
]
(1+o_{\frac{1}{\lambda_{i}}}(1))
$
 \item \quad 
$
\lambda_{i}\dot a_{i}
= 
\frac{r}{k}
[
\frac{d_{3}}{c_{3}}\frac{\nabla K_{i}}{K_{i}\lambda_{i}}
+
\frac{b_{3}}{c_{3}}
\sum_{i\neq j =1}^{p}\frac{\alpha_{j}}{\alpha_{i}}
\frac{1}{\lambda_{i}}\nabla_{a_{i}}\varepsilon_{i,j}
]
(1+o_{\frac{1}{\lambda_{i}}}(1))
$
\end{enumerate}
up to some $o_{\varepsilon}(\lambda_{i}^{\frac{2-n}{2}}+\sum_{i\neq j=1}^{p}\varepsilon_{i,j})$
 and 
 \begin{equation*}
O
(
\sum_{r\neq s}\frac{\vert \nabla K_{r}\vert^{2}}{\lambda_{r}^{2}}
+
\lambda_{r}^{2-n}
+
\varepsilon_{r,s}^{2}
)
+ 
\begin{cases}
O
(
\vert \delta J(u)\vert^{2}
)
\; \text{ under } \eqref{Yamabe_flow} 
\\
O
(
\vert \partial  J(u)\vert^{2}
)
\; \text{ under } \eqref{gradient_flow} 
\end{cases}.
\end{equation*}
\end{proposition}

The statements concerning the Yamabe type flow  \eqref{Yamabe_flow} are exactly those of Corollaries \ref{I-cor_simplifying_the_shadow_flow},\ref{I-cor_simplifying_the_shadow_flow_w} in \cite{may-cv}
and they are proven by testing the flow via
$\langle \partial_{t}u,\phi_{l,j}\rangle$. In case of \eqref{Yamabe_flow} the natural scalar product is
\begin{equation*}
\langle a,b\rangle_{Ku^{\frac{4}{n-2}}}=\int Ku^{\frac{4}{n-2}}ab.
\end{equation*}
Hence letting   $\dot \xi_{k,i}=(\frac{\dot \alpha_{i}}{\alpha_{i}},-\frac{\dot\lambda_{i}}{\lambda_{i}},\lambda_{i}\dot a_{i})$ we have to evaluate on $V(p,\varepsilon)$ under \eqref{Yamabe_flow} for instance
\begin{equation*}
\begin{split}
I_{1}+I_{2}
= &
\alpha_{i}\langle \phi_{k,i},\phi_{l,j}\rangle_{Ku^{\frac{4}{n-2}}}\dot \xi^{k,i} 
+
\langle \partial_{t}v,\phi_{l,j}\rangle_{Ku^{\frac{4}{n-2}}}
=
\langle \partial_{t}u,\phi_{l,j}\rangle_{Ku^{\frac{4}{n-2}}} \\
= &
-\langle (\frac{R}{K}-\frac{r}{k})u,\phi_{l,j}\rangle_{Ku^{\frac{4}{n-2}}}
=
I_{3},
\end{split}
\end{equation*}
where 
\begin{enumerate}[label=(\roman*)]
 \item \quad 
 $
\int Ku^{\frac{4}{n-2}}\phi_{k,i}\phi_{l,j}
=
c_{k}\alpha_{i}^{\frac{4}{n-2}}K_{i}\delta_{k,l}\delta_{i,j} 
 $
 up to some 
 \begin{equation*}
 O(
\frac{\vert \nabla K_{i}\vert}{\lambda_{i}}
+
\frac{1}{\lambda_{i}^{2}}
+
\frac{1}{\lambda_{i}^{n-2}} 
)\delta_{i,j}
+
O
(
\sum_{i\neq m=1}^{p}\varepsilon_{i,m} +\Vert v\Vert),
 \end{equation*}
\quad cf. the proof of Lemma \ref{I-lem_the_shadow_flow}   in \cite{may-cv}.
\item \quad    
$I_{2}
=
\int Ku^{\frac{4}{n-2}}\partial_{t} v\phi_{l,j}
=
-\int K\partial_{t}u^{\frac{4}{n-2}}v\phi_{l.j} 
+
O(\Vert v \Vert)_{k,l}\delta_{i,j}\dot \xi^{k,i} 
$ 
and 
\begin{equation*}
\int K\partial_{t}u^{\frac{4}{n-2}} v\phi_{l,j}
=
\frac{4}{n-2}\int (R-r\bar K)u^{\frac{4}{n-2}}v\phi_{l,j}
\end{equation*}
\item \quad 
$I_{3}=-\int (R-r\K)u^{\frac{n+2}{n-2}}\phi_{l,j} = -\frac{1}{2}\partial J(u)\phi_{l,j}$
, cf. Proposition \ref{prop_derivatives_of_J} and recalling $k=1$
\end{enumerate}
In contrast under \eqref{gradient_flow} the natural scalar product is 
\begin{equation*}
\langle a,b \rangle_{L_{g_{0}}}=\int L_{g_{0}}ab
\end{equation*}
and we have to evaluate 
\begin{equation*}
\begin{split}
I_{1}+I_{2}
= &
\alpha_{i}\langle \phi_{k,i},\phi_{l,j}\rangle_{L_{g_{0}}}\dot \xi^{k,i} 
+
\langle \partial_{t}v,\phi_{l,j}\rangle_{L_{g_{0}}} 
=
\langle \partial_{t}u,\phi_{l,j} \rangle_{L_{g_{0}}}
\\
= &
-\frac{r}{2k}\langle
\nabla J(u)
-
\frac{\int Ku^{\frac{n+2}{n-2}}\nabla J(u)}{k}u
,
\phi_{l,j}
\rangle_{L_{g_{0}} }
=
I_{3}, 
\end{split}
\end{equation*}
where 
\begin{enumerate}[label=(\roman*)]
 \item \quad  
 $
\int L_{g_{0}}\phi_{k,i}\phi_{l,j}
=
4n(n-1)c_{k}\delta_{k,l}\delta_{i,j} 
+
 O(
\frac{1}{\lambda_{i}^{2}}
+
\frac{1}{\lambda_{i}^{n-2}} 
)\delta_{i,j}
+
O
(
\varepsilon_{i,j})
 $
 \item \quad
 $I_{2}=\int L_{g_{0}}\partial_{t}v \phi_{l,j}=O(\Vert v \Vert)_{k,l}\delta_{i,j}\dot \xi^{k,i}$
 \item \quad
 $
I_{3}
=
-
\frac{r}{2k}\partial J(u)\phi_{l,j}
+
O
(
\int Ku^{\frac{n+2}{n-2}}\nabla J(u)
) 
 $
 and due to $\partial J(u)u=0$
 \begin{equation*}
\int Ku^{\frac{n+2}{n-2}}\nabla J(u)
=
O(\vert \partial J(u)\vert^{2}).
 \end{equation*}
\end{enumerate}
In order to compare (i)-(iii), note, that by virtue of Propositions \ref{prop_analysing_ski} we have 
\begin{equation}\label{small_difference}
K_{i}\alpha_{i}^{\frac{4}{n-2}}=4n(n-1)\frac{k}{r}
\end{equation}
up to some 
\begin{equation*}
 O
 ( 
 \sum_{r\neq s}
 \frac{\vert \nabla K_{r}\vert }{\lambda_{r}}+\frac{\vert \Delta K_{r}\vert}{\lambda_{r}^{2}}+\frac{1}{\lambda_{r}^{n-2}}+\varepsilon_{r,s}+\Vert v \Vert +\vert \partial J(u)\vert 
 ),
\end{equation*}
since 
$$
\sigma_{1,i}=O(\vert \partial J(u)\vert ), 
$$
cf. Proposition \ref{prop_analysing_ski}, also  (5.13) in \cite{may-cv} for the analogon in case $\omega \neq 0$. 
Consequently 
\begin{equation*}
 \Xi_{k,i,l,j}\dot \xi^{k,i}
=
\frac{r}{k}
\sigma_{l,j} 
+
\begin{cases}
O(\int (R-r\K)u^{\frac{4}{n-2}}v\phi_{l,j})\; \text{ under }\; \eqref{Yamabe_flow} \\
O(\vert \partial J(u)\vert^{2}) \; \text{ under }\; \eqref{gradient_flow}   
\end{cases}
\end{equation*}
with invertible 
\begin{equation}\label{Xi}
\begin{split}
\Xi_{k,i,l,j}
= &
4n(n-1)\alpha_{i}c_{k}\delta_{k,l}\delta_{i,j} +O(\frac{1}{\lambda_{i}^{2}})\delta_{i,j}\\
& +
 O
 ( 
 \sum_{r\neq s} 
 \frac{\vert \nabla K_{r}\vert }{\lambda_{r}}+\frac{\vert \Delta K_{r}\vert}{\lambda_{r}^{2}}+\frac{1}{\lambda_{r}^{n-2}}+\varepsilon_{r,s}+\Vert v \Vert +\vert \partial J(u)\vert 
 )
\end{split}
\end{equation}
and hence, since $\sigma_{l,j}=O(\vert \partial J(u)\vert)$
\begin{equation*}
\begin{split}
 \dot \xi_{k,i}
= &
\frac{\frac{r}{k}\sigma_{k,i} }{4n(n-1)\alpha_{i}c_{k}} (1+O(\frac{1}{\lambda_{i}^{2}})) \\
& +
\begin{cases}
O(\sum_{r}\int \vert R-r\K\vert u^{\frac{4}{n-2}}\vert v \vert \varphi_{r})\; \text{ under }\; \eqref{Yamabe_flow} \\
O(\vert \partial J(u)\vert^{2}) \; \text{ under }\; \eqref{gradient_flow}   
\end{cases} \\
& +
 O
 ( 
 \sum_{r\neq s}
 \frac{\vert \nabla K_{r}\vert ^{2}}{\lambda_{r}^{2}}+\frac{\vert \Delta K_{r}\vert^{2}}{\lambda_{r}^{4}}+\frac{1}{\lambda_{r}^{2(n-2)}}+\varepsilon_{r,s}^{2}+\Vert v \Vert^{2} +\vert \partial J(u)\vert^{2} 
 ).
\end{split}
\end{equation*}
Here enters the difference from \eqref{Yamabe_flow} to \eqref{gradient_flow}. In fact we have to estimate
\begin{equation*}
\begin{split}
\int \vert R  -r\K \vert u^{\frac{4}{n-2}}\vert v \vert \varphi_{r} 
\lesssim &
\int \vert R-r\K \vert u^{\frac{4}{n-2}}\vert v \vert (u+\vert v \vert) \\ 
\lesssim &
\Vert R-r\K \Vert_{L^{\frac{2n}{n+2}}_{g_{u}}}\Vert v \Vert
+
\Vert R-r\K \Vert_{L^{\frac{n}{2}}_{g_{u}}}\Vert v \Vert^{2}  \\
\leq &
\vert \delta J(u)\vert^{2}+ (1+\Vert R-r\K \Vert^{2}_{L^{\frac{n}{2}}_{g_{u}}})\Vert v \Vert^{2},
\end{split}
\end{equation*}
i.e. there appears $\vert \delta J(u)\vert$ instead of $\vert \partial J(u)\vert$. Also note, that we have 
\begin{equation*}
\Vert R-r\K \Vert_{L^{\frac{n}{2}}_{g_{u}}}\longrightarrow 0 \; \text{ as }\; t\longrightarrow \infty
\end{equation*}
along each flow line by virtue of Proposition \ref{I-prop_strong_convergence_of_the_first_variation} from \cite{may-cv}. We thus obtain 

\begin{equation*}
\begin{split}
 \dot \xi_{k,i}
= 
\frac{r}{k}\frac{\sigma_{k,i} }{4n(n-1)\alpha_{i}c_{k}} 
(1+o_{\frac{1}{\lambda_{i}}}(1))
 +
 O
 \biggr(&
 \sum_{r\neq s} 
 \frac{\vert \nabla K_{r}\vert ^{2}}{\lambda_{r}^{2}} 
+
 \frac{1}{\lambda_{r}^{4}}+\frac{1}{\lambda_{r}^{2(n-2)}}+\varepsilon_{r,s}^{2}\\
 & +
 \Vert v \Vert^{2} 
 +
\begin{cases}
\vert \delta J(u)\vert^{2}\; \text{ under }\; \eqref{Yamabe_flow} \\
\vert \partial J(u)\vert^{2} \; \text{ under }\; \eqref{gradient_flow}   
\end{cases}
\biggr).
\end{split}
\end{equation*}
Hence Proposition \ref{prop_the_shadow_flow} for $u\in V(p,\varepsilon)$ follows from Proposition \ref{prop_simplifying_ski}  and \eqref{small_difference}  absorbing $\Vert v \Vert$ via Proposition \ref{prop_a-priori_estimate_on_v}. The case $u\in V(\omega,p,\varepsilon)$ is analogous.

\subsection{Principal behaviour}\label{sec_principal_behaviour} 
Let us recall  some generic notions and results in the statements below. 

\begin{definition}\label{def_principally_lower_bounded_of_the_first_variation}
We call $\partial J$ principally lower bounded, 
if for every $p\geq 1$ there exist $c, \varepsilon>0$ such, that
\begin{enumerate}[label=(\roman*)]
 \item \quad
 $\vert \partial J(u)\vert 
\geq
c
(
\sum_{r}
\frac{\vert \nabla K_{r}\vert}{K_{r}\lambda_{r}}
+
\frac{\vert \lap K_{r}\vert }{K_{r}\lambda_{r}^{2}}
+
\lambda_{r}^{2-n}
+
\sum_{r\neq s}\varepsilon_{r,s}
)
$ for all $ u\in V(p, \varepsilon)$
 \item \quad
 $\vert \partial  J(u)\vert \geq
c
(
\sum_{r}
\frac{\vert \nabla K_{r}\vert}{K_{r}\lambda_{r}}
+
\lambda_{r}^{\frac{2-n}{2}}
+
\sum_{r\neq s}\varepsilon_{r,s}
)
$ for all $u\in V(\omega, p, \varepsilon).$
 \end{enumerate}
\end{definition}

Under this mild assumption  we have uniformity in $V(\omega, p, \varepsilon)$ as follows.

 \begin{proposition}
Assume $\partial J$ to be principally lower bounded. For 
$$u=u_{\alpha, \beta}+\alpha^{i}\var_{i}+v\in V(\omega, p, \varepsilon)$$ 
with $k_{u}\equiv 1$ we then have
\begin{equation*}\begin{split}
\lambda_{i}^{-1}, \varepsilon_{i,j}, \vert 1-\frac{r_{\infty}\alpha_{i}^{\frac{4}{n-2}}K_{i}}{4n(n-1)}\vert
,
\vert(\frac{r}{k})_{u_{1, \beta}}-r_{\infty}\alpha^{\frac{4}{n-2}}\vert, \vert \partial J(u_{1, \beta})\vert
,  
\Vert v \Vert
\longrightarrow 
0
\end{split}\end{equation*} 
uniformly as $\vert \partial J(u)\vert \longrightarrow 0$ and $J(u)=r\longrightarrow J_{\infty}=r_{\infty}$.
\end{proposition}
\begin{proof}
 Cf. Proposition \ref{I-prop_uniformity_in_V(omega,p,e)} in \cite{may-cv}.
\end{proof}

As a consequence we obtain limiting uniqueness of non compact flow lines in analogy to the unique limit of  compact flow lines. 
\begin{proposition}\label{prop_unicity_of_a_limiting_critical_point_at_infinity}

Assume $\partial J$ to be principally lower bounded.\\ 
If a sequence $u(t_{k})$ along \eqref{Yamabe_flow} or \eqref{gradient_flow} diverges 
in the sense, that
$$
\e p>1, \varepsilon_{k}\searrow 0\;:\; u(t_{k})\in V(\omega, p, \varepsilon_{k}),
$$
then $u$ diverges as well in the sense, that
$$
\e  p>1 \fa \varepsilon>0 \e T>0 \fa t>T\;:\; u(t)\in V(\omega, p, \varepsilon).
$$
\end{proposition} 
\begin{proof}
Cf. Proposition \ref{I-prop_unicity_of_a_limiting_critical_point_at_infinity} from \cite{may-cv}
\end{proof}
\begin{remark}
In the statement of Proposition \ref{prop_unicity_of_a_limiting_critical_point_at_infinity} and 
in contrast to its corresponding counterpart Proposition 
\ref{I-prop_unicity_of_a_limiting_critical_point_at_infinity} in \cite{may-cv} we have replaced 
\begin{center}
"\ldots converging to a critical point at infinity in the sense, that \ldots " 
\end{center}
by 
\begin{center}
"\ldots diverges in the sense, that\ldots ".  
\end{center}
In fact, as we have exposed in \cite{MM3} and will see in the present paper, not every non compact or diverging flow line leads to a critical point at infinity.  
\end{remark}
Note, that Proposition \ref{prop_unicity_of_a_limiting_critical_point_at_infinity} in combination with Proposition  \ref{prop_concentration_compactness} tells us, that every  non compact, i.e. diverging  flow line has to remain in some $V(\omega,p,\varepsilon)$ eventually for every $\varepsilon>0$. 
\begin{lemma}\label{lem_critical_points_of_K_as_attractors}
If $\partial J$ is principally lower bounded, then under \eqref{Yamabe_flow} or \eqref{gradient_flow} 
$$ 
K(a_{i})\longrightarrow K_{i_{\infty}} \; \text{ and } \; \vert \nabla K(a_{i})\vert\longrightarrow 0 \; \text{ for all }\;i=1, \ldots,p
$$
and every diverging flow line converges to a critical point at infinity.
\end{lemma}
\begin{proof}
 Cf. Proposition \ref{I-prop_unicity_of_a_limiting_critical_point_at_infinity} in \cite{may-cv}.
\end{proof}
Finally we note, that
\begin{proposition}\label{prop_princ_lower_bounded_under_Cond_n}
 $\partial J$ is principally lower bounded under Condition \ref{Condition_on_K}.
\end{proposition}

\begin{proof}
We just have to adapt the corresponding proof of Proposition 
\ref{I-prop_princ_lower_bounded_under_Cond_n} in \cite{may-cv} to this situation.  
In case $\omega =0$ Propositions
\ref{prop_a-priori_estimate_on_v}, 
\ref{prop_simplifying_ski}  and  
\eqref{small_difference} show 
\begin{enumerate}[label=(\roman*)]
 \item \quad 
$
\sigma_{2,i}
= 
\tilde \gamma_{1}\alpha_{i}\frac{ H_{i}}{\lambda_{i} ^{n-2}}
+
\gamma_{2}\alpha_{i}\frac{\lap K_{i}}{K_{i}\lambda_{i}^{2}} 
- 
\tilde \gamma_{5}
\sum_{i \neq j=1}^{p}\alpha_{j}
\lambda_{i}\partial_{\lambda_{i}}\varepsilon_{i,j} 
$
 \item \quad 
$
\sigma_{3,i}
= 
\tilde \gamma_{3}\alpha_{i}\frac{\nabla K_{i}}{K_{i}\lambda_{i}}
+
\gamma_{4}\alpha_{i}\frac{\nabla \lap K_{i}}{K_{i}\lambda_{i}^{3}} 
+
\gamma_{6}
\sum_{i \neq j=1}^{p}
\frac{\alpha_{j}}{\lambda_{i}}\nabla_{a_{i}}\varepsilon_{i,j} 
$
\end{enumerate}
up to some 
$
o_{\varepsilon}
(
\lambda_{i}^{2-n} 
+
\sum_{i\neq j=1}^{p}\varepsilon_{i,j}
)
$
and
\begin{equation*}
\begin{split}
O
(
\sum_{r} \frac{\vert \nabla K_{r}\vert^{2}}{\lambda_{r}^{2}} 
+
\frac{\vert \lap K_{r}\vert^{2}}{\lambda_{r}^{4}}
+
\frac{1}{\lambda_{r}^{2(n-2)}} +\sum_{r\neq s}\varepsilon_{r,s}^{2}
+
\vert \partial J(u)\vert^{2}
)
.                               
\end{split}
\end{equation*} 
Letting $0<\underline \kappa\leq \kappa_{i} \leq \overline\kappa <\infty$ for $\vert \nabla K_{i}\vert \neq 0$
and $\kappa_{i}=0$ for $\vert \nabla K_{i}\vert=0$ we get
\begin{equation}\label{testfunction_unicity>...}
\begin{split}
\sum_{i} & C^{i}(\frac{\sigma_{2,i}}{\alpha_{i}}
+\kappa_{i}\langle \frac{\sigma_{3,i}}{\alpha_{i}}, \frac{\nabla K_{i}}{\vert \nabla K_{i}\vert}\rangle ) \\
\geq &
\sum_{i}C^{i}
[
\gamma_{1}\frac{H_{i}}{\lambda_{i}^{n-2}}
+
\gamma_{2}\frac{\lap K_{i}}{K_{i}\lambda_{i}^{2}}
+
\gamma_{3}\kappa_{i}\frac{\vert \nabla K_{i}\vert}{K_{i}\lambda_{i}}
+
\gamma_{4}\kappa_{i}\frac{\langle \nabla \lap K_{i}, \nabla K_{i}\rangle}{K_{i}\vert \nabla K_{i}\vert\lambda_{i}^{3}}
] \\
& -
\tilde \gamma_{5}\sum_{i\neq j}C^{i}\frac{\alpha_{j}}{\alpha_{i}}
\lambda_{i}\partial_{\lambda_{i}} \varepsilon_{i,j}
+
o_{\varepsilon}(\sum_{r\neq s}\varepsilon_{r,s})
+
O
(
\sum_{i\neq j}\frac{C^{i}}{\lambda_{i}}\vert \nabla_{a_{i}}\varepsilon_{i,j}\vert
)\\
& +
O
(
\sum_{r}\frac{\vert \lap K_{r}\vert^{2}}{\lambda_{r}^{4}}
+
\vert \partial J(u)\vert^{2}
)
.
\end{split}
\end{equation} 
Ordering 
$
\frac{1}{\lambda_{1}}\geq \ldots\geq \frac{1}{\lambda_{p}}
$
we then have for 
$\varepsilon\ll 1$ and $C\gg 1$  
\begin{equation}\label{eij_large_unicity}
\begin{split}
-\sum_{i\neq j}C^{i}\frac{\alpha_{j}}{\alpha_{i}}\lambda_{i}\partial_{\lambda_{i}}\varepsilon_{i,j}
\geq c \sum_{i>j}C^{i}\varepsilon_{i,j}
\end{split}
\end{equation}
and
\begin{equation}\label{eij_small_unicity}
\begin{split}
\sum_{i\neq j}\frac{C^{i}}{\lambda_{i}}\vert \nabla_{a_{i}}\varepsilon_{i,j}\vert =O(\sum_{i>j}C^{j}\varepsilon_{i,j}).
\end{split}
\end{equation}

To prove \eqref{eij_large_unicity} and \eqref{eij_small_unicity} note, that
\begin{equation*}\begin{split}
& \sum_{i\neq j} C^{i}  \frac{\alpha_{j}}{\alpha_{i}}\lambda_{i}\partial_{\lambda_{i}}\varepsilon_{i,j} 
 = 
\sum_{i>j}
[
C^{i}\frac{\alpha_{j}}{\alpha_{i}}
-
C^{j}\frac{\alpha_{i}}{\alpha_{j}}
]
\lambda_{i}\partial_{\lambda_{i}}\varepsilon_{i,j} 
+
\sum_{i<j}C^{i}\frac{\alpha_{j}}{\alpha_{i}}
[
\lambda_{i}\partial_{\lambda_{i}}\varepsilon_{i,j}
+
\lambda_{j}\partial_{\lambda_{j}}\varepsilon_{i,j}
].
\end{split}\end{equation*}

One has
$
-
\lambda_{i}\partial_{\lambda_{i}}\varepsilon_{i,j}
-
\lambda_{j}\partial_{\lambda_{j}}\varepsilon_{i,j}=(n-2)\varepsilon_{i,j}^{\frac{n}{n-2}}\lambda_{i}\lambda_{j}\gamma_{n}G^{\frac{2}{2-n}}( a _{i}, a _{j})
>0
$
and 
\begin{equation}\begin{split}\label{interaction_sign}
-\lambda_{i}\partial_{\lambda_{i}}\varepsilon_{i,j}
= &
\frac{n-2}{2}\varepsilon_{i,j}^{\frac{n}{n-2}}
(
\frac{\lambda_{i} }{ \lambda_{j} }
-
\frac{ \lambda_{j} }{ \lambda_{i} }
+
\lambda_{i}\lambda_{j}\gamma_{n}G^{\frac{2}{2-n}}( a _{i}, a _{j}))\geq 
\frac{n-2}{4}\varepsilon_{i,j}
\end{split}\end{equation}

for $i>j$. Thus \eqref{eij_large_unicity} follows. Finally note, that
\begin{equation*}
\begin{split}
\sum_{i\neq j}\frac{C^{i}}{\lambda_{i}}\vert \nabla_{a_{i}}\varepsilon_{i,j}\vert 
= &
\frac{n-2}{2}\sum_{i< j}C^{i}\varepsilon_{i,j}
\vert
\frac
{(\frac{\lambda_{j}}{\lambda_{i}})^{\frac{1}{2}}(\lambda_{i}\lambda_{j})^{\frac{1}{2}}\gamma_{n}\nabla_{ a_{i}}G^{\frac{2}{2-n}}(a_{i},a_{j})}
{\frac{\lambda_{i}}{\lambda_{j}}+\frac{\lambda_{j}}{\lambda_{i}}
+
\lambda_{i}\lambda_{j}\gamma_{n}G^{\frac{2}{2-n}}(a_{i},a_{j})} 
\vert 
\end{split} 
\end{equation*}

up to some $o(\sum_{i\neq j}\varepsilon_{i,j})$, 
whence we immediately obtain \eqref{eij_small_unicity}.
\\
Plugging \eqref{eij_large_unicity} and \eqref{eij_small_unicity} into 
\eqref{testfunction_unicity>...} we obtain for $C>1$ sufficiently large
\begin{equation*}
\begin{split}
\sum_{i} & C^{i}(\frac{\sigma_{2,i}}{\alpha_{i}}
+\kappa_{i} \langle \frac{\sigma_{3,i}}{\alpha_{i}}, \frac{\nabla K_{i}}{\vert \nabla K_{i}\vert}\rangle ) \\
\geq &
\sum_{i}C^{i}
[
\gamma_{1}\frac{H_{i}}{\lambda_{i}^{n-2}}
+
\gamma_{2}\frac{\lap K_{i}}{K_{i}\lambda_{i}^{2}}
+
\gamma_{3}\kappa_{i}\frac{\vert \nabla K_{i}\vert}{K_{i}\lambda_{i}}
+
\gamma_{4}\kappa_{i}\frac{\langle \nabla \lap K_{i}, \nabla 
K_{i}\rangle}{K_{i}\vert \nabla K_{i}\vert \lambda_{i}^{3}}
] \\
& +
\gamma_{5}\sum_{i> j}C^{i}\varepsilon_{i,j}
+
O
(
\frac{\vert \lap K_{r}\vert^{2}}{\lambda_{r}^{4}}
+
\vert \partial  J(u)\vert^{2}
).
\end{split}
\end{equation*} 
In case $\lap K_{i}\geq 0$ or $\vert \nabla K_{i}\vert >\epsilon $ for $\epsilon >0$ small 
we immediately obtain
\begin{equation}\label{non_degeneracy_under_cond_n}
\begin{split} 
\gamma_{i}\frac{H_{i}}{\lambda_{i}^{n-2}}
+
\gamma_{2}\frac{\lap K_{i}}{K_{i}\lambda_{i}^{2}}
+
\gamma_{3}\kappa_{i}\frac{\vert \nabla K_{i}\vert}{K_{i}\lambda_{i}}
& +
\gamma_{4}\kappa_{i}\frac{\langle \nabla \lap K_{i}, \nabla K_{i}\rangle}{K_{i}\vert \nabla K_{i}\vert \lambda_{i}^{3}} \\
\geq & 
c
[
\frac{H_{i}}{\lambda_{i}^{n-2}}
+
\frac{\vert \lap K_{i}\vert}{K_{i}\lambda_{i}^{2}}
+
\frac{\vert \nabla K_{i}\vert}{K_{i}\lambda_{i}}
]
\end{split}
\end{equation} 
for some $c>0$ and all $\lambda_{i}>0$ sufficiently large choosing $\kappa_{i}$ such, that
\begin{equation*}
\begin{split}
\gamma_{i}\frac{H_{i}}{\lambda_{i}^{n-2}} 
+
\gamma_{4}\kappa_{i}\frac{\langle \nabla \lap K_{i}, \nabla K_{i}\rangle}{K_{i}\vert \nabla K_{i}\vert \lambda_{i}^{3}} 
\geq 
c\frac{H_{i}}{\lambda_{i}^{n-2}}.
\end{split}
\end{equation*} 
Also \eqref{non_degeneracy_under_cond_n} follows in case
$ \lap K_{i}<0$ and $\vert \nabla K_{i} \vert<\varepsilon $, unless
\begin{equation*}
d_{g_{0}}(a_{i},x_{0})\ll 1.
\end{equation*}
In particular \eqref{non_degeneracy_under_cond_n}  follows in case 
$ \lap K_{i}<0$ and $\vert \nabla K_{i} \vert=0$, since then by Condition \ref{Condition_on_K} 
\begin{equation*}
a_{i}=x_{0}\; \text{ and }\; \nabla K_{i}=0,\Delta K_{i}=0,\; \nabla \Delta K_{i}=0.
\end{equation*}
Finally in case 
$ \lap K_{i}<0$ and $0\neq \vert \nabla K_{i} \vert<\varepsilon$ we have 
\begin{equation*}
\langle\nabla \lap K_{i}, \nabla K_{i}\rangle
=
32(n+2) \vert a_{i} \vert^{4}
=\frac{2}{n+2}\vert \Delta K_{i}\vert^{2}
\end{equation*}
and thus by Cauchy-Schwarz inequality 
\begin{equation*}
\begin{split}
\frac{\lap K_{i}}{K_{i}\lambda_{i}^{2}}
> &
-\frac{1}{2}\sqrt{\frac{n+2}{2}}\frac{\vert \nabla K_{i}\vert}{K_{i}\lambda_{i}}
-\frac{1}{2}\sqrt{\frac{n+2}{2}}
\frac
{
\langle \nabla \lap K_{i}, \nabla K_{i}\rangle
}
{
K_{i}\vert \nabla K_{i}\vert \lambda_{i}^{3}
}.
\end{split}
\end{equation*} 
Choosing therefore $\kappa_{i}$ such, that 
\begin{equation*}
\frac{1}{2}\sqrt{\frac{n+2}{2}}\gamma_{2}<\gamma_{3}\kappa_{i}
\; \text{ and } \; 
\frac{1}{2}\sqrt{\frac{n+2}{2}}\gamma_{2}<\gamma_{4}\kappa_{i},
\end{equation*}
 then \eqref{non_degeneracy_under_cond_n} holds true as well
and thus in any case. We conclude
\begin{equation*}
\begin{split}
\sum_{i} & C^{i}(\frac{\sigma_{2,i}}{\alpha_{i}}
+\kappa_{i} \langle\frac{ \sigma_{3,i}}{\alpha_{i}}, \frac{\nabla K_{i}}{\vert \nabla K_{i}\vert}\rangle ) 
\geq 
\sum_{i}
[
\frac{H_{i}}{\lambda_{i}^{n-2}}
+
\frac{\vert \lap K_{i}\vert}{K_{i}\lambda_{i}^{2}}
+
\frac{\vert \nabla K_{i}\vert}{K_{i}\lambda_{i}}
] 
+
\sum_{i> j}\varepsilon_{i,j}
\end{split}
\end{equation*} 
up to some $O(\vert \partial J(u)\vert^{2})$. Since $\sigma_{k,i}=O(\vert \partial J(u)\vert)$ by definition, the claim follows noticing $H_{i}>c>0$ due to $M\not \simeq \mathbb{S}^n$ and  by means of the positive mass theorem.  
 The case $\omega >0$ is proven analogously.
\end{proof}

\section{Divergence and Compactification}\label{sec_divergence_and_compactification}
Throughout this section we assume Condition \ref{Condition_on_K} to hold true and identify the lack of compactness of the flows on $X$ generated by \eqref{Yamabe_flow} and \eqref{gradient_flow}. Subsequently will perform a slight variation of these flows and thereby restore compactness. 
\subsection{Compact regions}\label{sec_compact_regions}
In order to describe how non compact flow lines under \eqref{Yamabe_flow} or \eqref{gradient_flow} look like, we first exclude most of the generic  possibilities of diverging flow lines within $V(\omega,p,\varepsilon)$, since 
by virtue of Propositions \ref{prop_concentration_compactness}  and  \ref{prop_unicity_of_a_limiting_critical_point_at_infinity} we know, that every non compact flow line has to remain in some $V(\omega,p,\varepsilon)$ eventually, provided $ \partial J $ is principally lower bounded, cf. Definition \ref{def_principally_lower_bounded_of_the_first_variation} and this we ensure
by Condition \ref{Condition_on_K} via 
 Proposition \ref{prop_princ_lower_bounded_under_Cond_n}. Moreover 
Lemma \ref{lem_critical_points_of_K_as_attractors} then allows us to distinguish non compact flow lines with respect to their end configuration. In fact, since we assume 
$
\{x_{0},\ldots,x_{q}\}=\{\nabla K=0\}
$
and there holds 
$$\vert \nabla K_{i}\vert =\vert \nabla K(a_{i})\vert \longrightarrow 0 \; \text{ as } \; t\longrightarrow \infty
$$ 
by virtue of Lemma \ref{lem_critical_points_of_K_as_attractors}, we find 
$a_{i}\longrightarrow x_{j_{i}}$ as $t\longrightarrow \infty$.
\begin{lemma}\label{lem_no_mixed_type}
Every non zero weak limit flow line, i.e. eventually
\begin{equation*}
u \not \in V(p,\varepsilon),
\end{equation*}
is compact. 
\end{lemma}
\begin{proof}
Since every flow line constitutes up to a subsequence in time a Palais-Smale sequence, cf. \eqref{flow_palais_smale_yamabe_flow} and \eqref{flow_palais_smale_gradient}, Propositions \ref{prop_concentration_compactness} and \ref{prop_unicity_of_a_limiting_critical_point_at_infinity} tell us, that we may assume $u\in V(\omega,p,\varepsilon)$ for all times to come for some $V(\omega,p,\varepsilon)$ and $u\longrightarrow \omega$ strongly in case $\omega>0$ and $p=0$, in which case $u$ as a flow line is compact. Hence we may assume, that eventually $u\in V(\omega,p,\varepsilon)$ for $\omega>0$ and $p\geq 1$.  Then Proposition \ref{prop_the_shadow_flow} and the principal lower bound on $\partial J$, cf. Definition \ref{def_principally_lower_bounded_of_the_first_variation}, give 
$$
-\frac{\dot \lambda_{i}}{\lambda_{i}}
= 
\frac{r}{k}
[
\frac{d_{2}}{c_{2}}
\frac{\alpha \omega_{i}}{\alpha_{i}K_{i}\lambda_{i} ^{\frac{n-2}{2}}}
-
\frac{b_{2}}{c_{2}}\sum_{i\neq j =1}^{p}\frac{\alpha_{j}}{\alpha_{i}}
\lambda_{i}\partial_{\lambda_{i}}\varepsilon_{i,j}
](1+o_{\frac{1}{\lambda_{i}}}(1)) 
$$
up to some
\begin{equation*}  
o_{\varepsilon}(\lambda_{i}^{\frac{2-n}{2}}+\sum_{i\neq j=1}^{p}\varepsilon_{i,j})
+
\begin{cases} 

O
(
\vert \delta J(u)\vert^{2}
)
\; \text{ under } \eqref{Yamabe_flow} 
\\

O
(
\vert \partial  J(u)\vert^{2}
)
\; \text{ under } \eqref{gradient_flow} 
\end{cases}.
\end{equation*}
Then ordering 
$
\frac{1}{\lambda_{1}}\geq \ldots\geq \frac{1}{\lambda_{p}}
$
and recalling  \eqref{eij_large_unicity} and $\omega_{i}=\omega(a_{i})>0$ we find for   
$
\psi=\sum_{i}C^{i}\ln\frac{1}{ \lambda_{i}}
$
\begin{equation*}
 \psi^{\prime} \geq 
 \begin{cases}
O
(
\vert \delta J(u)\vert^{2}
)
\; \text{ under } \eqref{Yamabe_flow} 
\\
O
(
\vert \partial  J(u)\vert^{2}
)
\; \text{ under } \eqref{gradient_flow} 
\end{cases}.
\end{equation*}
Then the right hand side is integrable in time, while necessarily  $\psi \longrightarrow -\infty$
as some $\lambda_{i}\longrightarrow\infty$. Hence all  $\lambda_{i}$ have to stay bounded, which due to the principal lower bound on $\partial J$ prevents $\vert \partial J(u)\vert\longrightarrow 0,$ hence contradicting the time  integrability of
$\vert \partial J(u)\vert^{2}$. 
\end{proof}

\begin{lemma}\label{lem_no_concentration_away_from_x_0}
Every flow line away from $x_{0}$, i.e. eventually
\begin{equation*}
u \not \in V(p,\varepsilon)\cap \{ \forall \;1\leq i \leq p\;:\; a_{i}\xrightarrow{t\to \infty}  x_{0}\},
\end{equation*}
is compact. 
\end{lemma}
\begin{proof}
We may assume  $u\in V(p,\varepsilon)$ eventually.  Then Proposition \ref{prop_the_shadow_flow} and the principal lower bound on $\partial J$, cf. Definition \ref{def_principally_lower_bounded_of_the_first_variation}, give 
$$
-\frac{\dot \lambda_{i}}{\lambda_{i}}
= 
\frac{r}{k}
[
\frac{d_{2}}{c_{2}}
\frac{ H_{i}}{\lambda_{i} ^{n-2}}
+
\frac{e_{2}}{c_{2}}\frac{\lap K_{i}}{K_{i}\lambda_{i} ^{2}} 
-
\frac{b_{2}}{c_{2}} \sum_{i \neq j=1}^{p}\frac{\alpha_{j}}{\alpha_{i}}
\lambda_{i}\partial_{\lambda_{i}}\varepsilon_{i,j} 
]
(1+o_{\frac{1}{\lambda_{i}}}(1)) 
$$
up to some
\begin{equation*}
o_{\varepsilon}  ( \lambda_{i}^{2-n}  +\sum_{i\neq j=1}^{p}\varepsilon_{i,j})
+
\begin{cases}
O
(
\vert \delta J(u)\vert^{2}
)
\; \text{ under } \eqref{Yamabe_flow} 
\\

O
(
\vert \partial  J(u)\vert^{2}
)
\; \text{ under } \eqref{gradient_flow} 
\end{cases}.
\end{equation*}
Moreover by assumption
\begin{equation*}
\{1,\ldots,p\}=P\supseteq Q=\{1\leq i \leq p\;:\; a_{i}\xrightarrow{t\to \infty} x_{j_{i}}\neq x_{0}\}\neq \emptyset. 
\end{equation*} 
Then ordering 
$\frac{1}{\lambda_{l_{1}}}\geq \ldots \geq \frac{1}{\lambda_{l_{q}}}$ for 
$Q=\{l_{1},\ldots,l_{q}\}$ we  consider
\begin{equation*}
\psi=\sum_{i=1}^{q}C^{i}\ln \frac{1}{\lambda_{l_{i}}}.
\end{equation*}
Since $\Delta K_{l_{i}}>0$ for $l_{i}\in Q$, as
$
\{\Delta K\leq 0\}\cap \{\nabla K=0\}=\{x_{0}\} ,
$
we have  
\begin{equation*}
\psi^{\prime}
\geq c\sum_{i=1}^{q}\frac{1}{\lambda_{l_{i}}^{2}}
-
\frac{r}{k}\frac{b_{2}}{c_{2}}
\sum_{Q\ni l_{i} \neq j \in P }C^{i}\frac{\alpha_{j}}{\alpha_{l_{i}}}\lambda_{l_{i}}\partial_{\lambda_{l_{i}}}\varepsilon_{l_{i},j} 
+
\begin{cases}
O
(
\vert \delta J(u)\vert^{2}
)
\; \text{ under } \eqref{Yamabe_flow} 
\\

O
(
\vert \partial  J(u)\vert^{2}
)
\; \text{ under } \eqref{gradient_flow} 
\end{cases}.
\end{equation*}
Recalling \eqref{eij_large_unicity} we then find
\begin{equation*}
-\sum_{Q\ni l_{i} \neq j \in Q }C^{i}\frac{\alpha_{j}}{\alpha_{l_{i}}}\lambda_{l_{i}}\partial_{\lambda_{l_{i}}}\varepsilon_{l_{i},j} 
\geq c\sum_{Q\ni l_{i}\neq j \in Q}\varepsilon_{l_{i},j}
\end{equation*}
and secondly
\begin{equation*}
-\sum_{Q\ni l_{i} \neq j \in P\setminus Q }C^{i}\frac{\alpha_{j}}{\alpha_{l_{i}}}\lambda_{l_{i}}\partial_{\lambda_{l_{i}}}\varepsilon_{l_{i},j} 
\geq 
c\sum_{Q\ni l_{i} \neq j \in P\setminus Q }\varepsilon_{l_{i},j},
\end{equation*}
since for $l_{i}\in Q$ and $j\in P\setminus Q$ by definition 
\begin{equation*}
d(a_{l_{i}},x_{j_{i}})\ll 1\; \text{ for some } x_{j_{i}}\neq x_{0},
\; \text{ while }\; 
d(a_{j},x_{0})\ll 1,
\end{equation*}
hence $a_{l_{i}}$ and $a_{j}$ are far from each other and therefore, cf. Lemma \ref{lem_interactions}, 
\begin{equation*}
\lambda_{l_{i}}\partial_{\lambda_{l_{i}}}\varepsilon_{l_{i},j}
=
\frac{2-n}{2}\varepsilon_{l_{i},j}^{\frac{n}{n-2}}
(
\frac{\lambda_{l_{i}}}{\lambda_{j}}-\frac{\lambda_{j}}{\lambda_{l_{i}}}
+
\lambda_{l_{i}}\lambda_{j} \gamma_{n}G_{g_{0}}^{\frac{2}{2-n}}(a_{l_{i}},a_{j})
)
=
\frac{2-n}{2}\varepsilon_{l_{i},j}(1+o(1)).
\end{equation*}
Hence, while $\psi \longrightarrow -\infty$ as some $\lambda_{l_{i}}\longrightarrow \infty$, we have
\begin{equation*}
\psi^{\prime }
\geq 
\begin{cases}
O
(
\vert \delta J(u)\vert^{2}
)
\; \text{ under } \eqref{Yamabe_flow} 
\\

O
(
\vert \partial  J(u)\vert^{2}
)
\; \text{ under } \eqref{gradient_flow}
\end{cases}
\end{equation*}
in contradiction to, that necessarily $\lambda_{i}\xrightarrow{t\to\infty}\infty$. 
\end{proof}
Lemma \ref{lem_no_concentration_away_from_x_0} tells us, that every diverging flow line can only  concentrate at $x_{0}=\max_{M}K$. We now exclude tower bubbling at $x_{0}$ as well. 
\begin{lemma}\label{lem_no_tower_bubbling_at_x_0}
Every non single bubbling flow line at $x_{0}$, i.e. 
\begin{equation*}
u \not \in V(1,\varepsilon)\cap \{a\xrightarrow{t\to \infty}  x_{0}\},
\end{equation*}
is compact. 
\end{lemma}
\begin{proof}
We may assume $u\in V(p,\varepsilon)$ eventually and  $\forall_{i} \,a_{i}\longrightarrow x_{0}$.   Then Proposition \ref{prop_the_shadow_flow} and the principal lower bound on $\partial J$, cf. Definition \ref{def_principally_lower_bounded_of_the_first_variation}, give

\begin{enumerate}[label=(\roman*)] 
 \item \quad 
$
-\frac{\dot \lambda_{i}}{\lambda_{i}}
= 
\frac{r}{k}
[
\frac{d_{2}}{c_{2}}
\frac{ H_{i}}{\lambda_{i} ^{n-2}}
+
\frac{e_{2}}{c_{2}}\frac{\lap K_{i}}{K_{i}\lambda_{i} ^{2}} 
-
\frac{b_{2}}{c_{2}} \sum_{i \neq j=1}^{p}\frac{\alpha_{j}}{\alpha_{i}}
\lambda_{i}\partial_{\lambda_{i}}\varepsilon_{i,j} 
]
(1+o_{\frac{1}{\lambda_{i}}}(1)) 
$
 \item \quad 
$
\lambda_{i}\dot a_{i}
= 
\frac{r}{k}
[
\frac{e_{3}}{c_{3}}  \frac{\nabla K_{i}}{K_{i}\lambda_{i}}
+
\frac{e_{4}}{c_{3}}  \frac{\nabla \lap K_{i}}{K_{i}\lambda_{i}^{3}}
+
\frac{b_{3}}{c_{3}} \sum_{i \neq j=1}^{p}\frac{\alpha_{j}}{\alpha_{i}}
 \frac{1}{\lambda_{i}}\nabla_{a_{i}}\varepsilon_{i,j}
](1+o_{\frac{1}{\lambda_{i}}}(1)) 
$
\end{enumerate}
up to some
\begin{equation*}
o_{\varepsilon}  ( \lambda_{i}^{2-n}  +\sum_{i\neq j=1}^{p}\varepsilon_{i,j})
+ 
\begin{cases}
O
(
\vert \delta J(u)\vert^{2}
)
\; \text{ under } \eqref{Yamabe_flow} 
\\
O
(
\vert \partial  J(u)\vert^{2}
)
\; \text{ under } \eqref{gradient_flow} 
\end{cases}.
\end{equation*}
More precisely by Condition \ref{Condition_on_K} 
and  recalling $K_{i}=K(a_{i})$ et cetera we have
\begin{equation*}\begin{split}
\nabla K_{i}=-4\vert a_{i}\vert^{2}a_{i},
\,
\Delta K_{i}=-4\cdot 7\vert a_{i} \vert^{2}
\; \text{ and }\;
\nabla\Delta K_{i}=-8\cdot 7 a_{i}.
\end{split}\end{equation*}
Consequently putting 
$\frac{d_{2}}{c_{2}}=\gamma_{1},\; \frac{e_{2}}{c_{2}}=\gamma_{2},\; \gamma_{3}=\frac{e_{3}}{c_{3}}$
and $b=\frac{b_{2}}{c_{2}}$ we find
\begin{enumerate}[label=(\roman*)] 
 \item \quad 
$
-\frac{\dot \lambda_{i}}{\lambda_{i}}
= 
\frac{r}{k}
[
\gamma_{1}
\frac{ H_{i}}{\lambda_{i} ^{n-2}}
-
4\cdot 7
\gamma_{2}\frac{\vert a _{i} \vert^{2}}{\lambda_{i} ^{2}} 
-
b \sum_{i \neq j=1}^{p}\frac{\alpha_{j}}{\alpha_{i}}
\lambda_{i}\partial_{\lambda_{i}}\varepsilon_{i,j} 
]
(1+o_{\frac{1}{\lambda_{i}}+\vert a_{i}\vert}(1)) 
$
 \item \quad 
$ 
\lambda_{i}\dot a_{i}
= 
\frac{r}{k}
[
-4\gamma_{3}\frac{\vert a_{i}\vert^{2}a_{i}}{\lambda_{i}}
+
O
( 
\sum_{i \neq j=1}^{p}
\vert \frac{\nabla_{a_{i}}}{\lambda_{i}}\varepsilon_{i,j}\vert
)
](1+o_{\frac{1}{\lambda_{i}}+\vert a_{i}\vert}(1)) 
$
\end{enumerate}
up to some
\begin{equation*}
o_{\varepsilon}  ( \lambda_{i}^{2-n}  +\sum_{i\neq j=1}^{p}\varepsilon_{i,j})
+ 
\begin{cases}
O
(
\vert \delta J(u)\vert^{2}
)
\; \text{ under } \eqref{Yamabe_flow} 
\\

O
(
\vert \partial  J(u)\vert^{2}
)
\; \text{ under } \eqref{gradient_flow} 
\end{cases}.
\end{equation*}
We first order 
$\lambda_{1}\vert a_{1}\vert^{5}\leq \ldots \leq  \lambda_{p}\vert a_{p}\vert^{5}$ and study  for $C \gg 1\gg \epsilon>0$
\begin{equation}
\Theta 
=
\sum_{i}C^{i}\eta(\frac{\lambda_{i}\vert a_{i}\vert^{5}}{\epsilon})\ln\frac{\lambda_{i}\vert a_{i}\vert^{5}}{\epsilon}
\end{equation} 
with a cut-off function $\eta\in C^{\infty}(\R,[0,1])$ satisfying 
\begin{equation*}
\eta\lfloor_{(0,1)}=0,\; \eta\lfloor_{(2,\infty)}=1\; \text{ and }\; \eta^{\prime}\lfloor_{(1,2)}>0. 
\end{equation*}
Then clearly $\Theta\geq 0$ and there holds
\begin{equation*}
\Theta^{\prime}
=
\sum_{i}C^{i}\vartheta_{i}\partial_{t}\ln(\lambda_{i}\vert a_{i}\vert^{5}),
\end{equation*}
where 
\begin{equation}\label{vartheta}
\vartheta_{i}
=
\eta(\frac{\lambda_{i}\vert a_{i}\vert^{5}}{\epsilon})
+
\eta^{\prime}(\frac{\lambda_{i}\vert a_{i}\vert^{5}}{\epsilon})
\frac{\lambda_{i}\vert a_{i}\vert^{5}}{\epsilon}
\ln \frac{\lambda_{i}\vert a_{i}\vert^{5}}{\epsilon}
\end{equation} and hence
\begin{equation*}
\vartheta_{i}
\; \text{ is }\; 
\begin{cases}
=0\; \text{ on }\;\lambda_{i}\vert a_{i}\vert^{5}\leq \epsilon \\
>0\; \text{ on }\; \epsilon<\lambda_{i}\vert a_{i}\vert^{5}\leq 2\epsilon \\
= 1\; \text{ on }\; \lambda_{i}\vert a_{i}\vert^{5}\geq 2\epsilon
\end{cases}.
\end{equation*} 
We then find 
\begin{equation*}
\begin{split}
\Theta^{\prime}
= &
\sum_{i}C^{i}\vartheta_{i}
(
\frac{\dot \lambda_{i}}{\lambda_{i}}
+
5\vert a_{i}\vert^{-2}\langle \frac{a_{i}}{\lambda_{i}},\lambda_{i}\dot a_{i}\rangle ) \\
\leq  &
\frac{r}{k}
\sum_{i}C^{i}\vartheta_{i}(1+o_{\varepsilon}(1))  
(
(
4\cdot 7
\gamma_{2}
-
5\cdot 4 \gamma_{3}
)
\frac{\vert a _{i} \vert^{2}}{\lambda_{i} ^{2}} \\
& \quad\quad\quad\quad\quad\quad\quad\quad\quad\quad\quad\quad\;\;\, +
b \sum_{i \neq j=1}^{p}\frac{\alpha_{j}}{\alpha_{i}}
\lambda_{i}\partial_{\lambda_{i}}\varepsilon_{i,j} 
+
o_{\varepsilon}  ( \sum_{i\neq j=1}^{p}\varepsilon_{i,j})
)  
\end{split}
\end{equation*}
up to some
\begin{equation*} 
\begin{cases}
O
(
\vert \delta J(u)\vert^{2}
)
\; \text{ under } \eqref{Yamabe_flow} 
\\

O
(
\vert \partial  J(u)\vert^{2}
)
\; \text{ under } \eqref{gradient_flow} 
\end{cases}.
\end{equation*}
Due to  $\frac{\gamma_{3}}{\gamma_{2}}=3$, cf.  the proof of Proposition  \ref{I-prop_n=5} in \cite{may-cv},
we have 
\begin{equation*}
4\cdot 7 \gamma_{2}-5\cdot 4 \gamma_{3}=-32\gamma_{2}
\end{equation*}
and there holds, cf. \eqref{interaction_sign} and arguing as for \eqref{eij_large_unicity}, for $i>j$
\begin{equation}\label{interactions_large_perturbed_case}
-\vartheta_{i}\lambda_{i}\partial_{\lambda_{i}}\varepsilon_{i,j} 
\geq c\vartheta_{i}\varepsilon_{i,j}
\; \text{ and  }\;  
-\sum_{i \neq  j}C^{i}\vartheta_{i}\frac{\alpha_{j}}{\alpha_{i}}
\lambda_{i}\partial_{\lambda_{i}}\varepsilon_{i,j} 
\geq
c\sum_{i\neq j}C^{i}\vartheta_{i}\varepsilon_{i,j},
\end{equation} 
as we shall prove below. 
We thus obtain
\begin{equation}\label{Theta_prime}
\begin{split}
\Theta^{\prime}
\leq  & 
-
c
\sum_{i}C^{i}\vartheta_{i}
(
\frac{\vert a _{i} \vert^{2}}{\lambda_{i} ^{2}} 
+
\sum_{i \neq j=1}^{p} \varepsilon_{i,j} 
)  
+
\begin{cases}
O
(
\vert \delta J(u)\vert^{2}
)
\; \text{ under } \eqref{Yamabe_flow} 
\\
O
(
\vert \partial  J(u)\vert^{2}
)
\; \text{ under } \eqref{gradient_flow} 
\end{cases}.
\end{split}
\end{equation}
As a consequences  $\Theta$, hence all $\lambda_{i}\vert a_{i}\vert^{5}$ are bounded and  
\begin{equation}\label{integrability_condition_tower_bubbles}
\forall\; 1\leq i \leq p\;:\;\int^{\infty}_{t=0}(\frac{\vert a_{i}\vert^{2}}{\lambda_{i}^{2}}+\sum_{i\neq j=1}^{p}\varepsilon_{i,j})\chi_{\{\lambda_{i}\vert a_{i}\vert^{5}\geq 2\varepsilon\}}
< 
\infty. 
\end{equation}
On the other hand for all $1\leq i \leq p$
\begin{equation*}
\frac{\dot \lambda_{i}}{\lambda_{i}}
\lesssim 
\frac{\vert a _{i} \vert^{2}}{\lambda_{i} ^{2}} 
+
\sum_{i \neq j=1}^{p}
\varepsilon_{i,j} 
+
\begin{cases}
O
(
\vert \delta J(u)\vert^{2}
)
\; \text{ under } \eqref{Yamabe_flow} 
\\

O
(
\vert \partial  J(u)\vert^{2}
)
\; \text{ under } \eqref{gradient_flow} 
\end{cases},
\end{equation*}
whence $\lambda_{i}\longrightarrow \infty$ due to \eqref{integrability_condition_tower_bubbles} necessitates, that for some $t_{k,i}\xrightarrow{k\to \infty}\infty$ at least
\begin{equation*}
\lambda_{i}\vert a_{i}\vert^{5} \leq 2\varepsilon \; \text{ at }\; t=t_{k,i},
\end{equation*}
while arguing as before on $\{\lambda_{i}\vert a_{i}\vert^{5}\geq 2\varepsilon\}$ 
\begin{equation*}
\partial_{t}\ln(\lambda_{i}\vert a_{i}\vert^{5})
\lesssim
\frac{\vert a_{i}\vert^{2}}{\lambda_{i}^{2}}+\sum_{i\neq j =1}^{\infty}\varepsilon_{i,j}
+
\begin{cases}
O
(
\vert \delta J(u)\vert^{2}
)
\; \text{ under } \eqref{Yamabe_flow} 
\\

O
(
\vert \partial  J(u)\vert^{2}
)
\; \text{ under } \eqref{gradient_flow} 
\end{cases}.
\end{equation*}
Hence we may assume, that eventually 
$\forall\; 1\leq i \leq p\;:\; \lambda_{i}\vert a_{i}\vert^{5}\leq 4\varepsilon$, 
 thus
\begin{equation*}
\begin{split}
\varepsilon_{i,j}
\gtrsim &
(\frac{\lambda_{i}}{\lambda_{j}}+\frac{\lambda_{j}}{\lambda_{i}}+\lambda_{i}\lambda_{j}\vert a_{i}-a_{j}\vert^{2})^{\frac{2-n}{2}}
\gtrsim 
(\frac{\lambda_{i}}{\lambda_{j}}+\frac{\lambda_{j}}{\lambda_{i}}+\lambda_{i}\lambda_{j}(
\vert a_{i}\vert^{2} + \vert a_{j}\vert^{2}))^{-\frac{3}{2}} \\
\gtrsim &
(\frac{\lambda_{i}}{\lambda_{j}}+\frac{\lambda_{j}}{\lambda_{i}}
+
\lambda_{i}\lambda_{j}
(
\frac{\lambda_{i}\vert a_{i}\vert^{5}}{\lambda_{i}} 
+
\frac{\lambda_{j} \vert a_{j}\vert^{5}}{\lambda_{j}}
)^{\frac{2}{5}})^{-\frac{3}{2}} 
\gtrsim
\varepsilon^{-\frac{3}{5}}
(
\lambda_{i}^{1-\frac{2}{5}}\lambda_{j}
+
\lambda_{i}\lambda_{j}^{1-\frac{2}{5}}
)^{-\frac{3}{2}} 
\end{split}
\end{equation*}
and likewise
$
\frac{\vert a_{i}\vert^{2}}{\lambda_{i}^{2}}\leq \frac{\varepsilon^{\frac{2}{5}}}{\lambda_{i}^{2+\frac{2}{5}}}.
$
Recalling \eqref{interaction_sign} we therefore obtain for $\lambda_{m}=\max_{i}\lambda_{i}$  
\begin{equation*}
\begin{split}
\frac{\dot \lambda_{m}}{\lambda_{m}}
\leq &  
\tilde \gamma_{2}\frac{\vert a _{m} \vert^{2}}{\lambda_{m} ^{2}} 
-
\tilde \gamma_{4} \sum_{m \neq j=1}^{p}\varepsilon_{m,j}      
\lesssim
\hat\gamma_{2}
\frac{\varepsilon^{\frac{2}{5}}}{\lambda_{m} ^{2+\frac{2}{5}}} 
-
\hat \gamma_{4}
 \sum_{m \neq j=1}^{p}
\frac{\varepsilon^{-\frac{3}{5}}}{(\lambda_{j}^{1-\frac{2}{5}}\lambda_{m}+\lambda_{j}\lambda_{m}^{1-\frac{2}{5}})^{\frac{3}{2}}} \\
\lesssim &
\bar \gamma_{2}
\frac{\varepsilon^{\frac{2}{5}}}{\lambda_{m} ^{2+\frac{2}{5}}} 
-
\bar \gamma_{4}
\frac{\varepsilon^{-\frac{3}{5}}}{\lambda_{m}^{\frac{3}{2}(2-\frac{2}{5})}}
= 
\frac{\check \gamma_{2}\varepsilon^{\frac{2}{5}}-\check \gamma_{4}\varepsilon^{-\frac{3}{5}}}{\lambda_{m} ^{2+\frac{2}{5}}} 
\leq 0
\end{split}
\end{equation*}  
up to some 
$$\begin{cases}
O
(
\vert \delta J(u)\vert^{2}
)
\; \text{ under } \eqref{Yamabe_flow} 
\\
O
(
\vert \partial  J(u)\vert^{2}
)
\; \text{ under } \eqref{gradient_flow}  
\end{cases}.
$$
So $\lambda_{m}\longrightarrow \infty$ is impossible and we are left with proving
\eqref{interactions_large_perturbed_case}. Recalling 
\begin{equation*}
\lambda_{1}\vert a_{1}\vert^{5}\leq \ldots \leq  \lambda_{p}\vert a_{p}\vert^{5}
\end{equation*}
we have for $i>j$
\begin{equation*}\begin{split}
-\lambda_{i}\partial_{\lambda_{i}}\varepsilon_{i,j}
= &
\frac{n-2}{2}\varepsilon_{i,j}^{\frac{n}{n-2}}
(
\frac{\lambda_{i} }{ \lambda_{j} }
-
\frac{ \lambda_{j} }{ \lambda_{i} }
+
\lambda_{i}\lambda_{j}\gamma_{n}G^{\frac{2}{2-n}}( a _{i}, a _{j}))
\end{split}\end{equation*}
and hence $-\lambda_{i}\partial_{\lambda_{i}}\varepsilon_{i,j}\geq \frac{n-2}{4}\varepsilon_{i,j}$ in either of the cases
\begin{equation*}
\lambda_{i}\geq \lambda_{j}
\; \text{ or } \;
\lambda_{i}\lambda_{j}d^{2}_{g_{0}}(a_{i},a_{j})\geq \frac{\lambda_{j}}{\lambda_{i}}.
\end{equation*}
Hence we may assume $d_{g_{0}}(a_{i},a_{j})\leq \frac{1}{\lambda_{i}}$ and 
$\frac{\lambda_{j}}{\lambda_{i}}\gg 1$.  Since for $i>j$ by assumption
\begin{equation*}
\lambda_{i}\vert a_{i}\vert^{5}\geq \lambda_{j}\vert a_{j}\vert^{5},
\end{equation*}
we then have $\vert a_{i}\vert \gg \vert a_{j}\vert$ and hence 
$d_{g_{0}}(a_{i},a_{j})\simeq \vert a_{i}-a_{j}\vert \simeq \vert a_{i}\vert$.   Therefore 
\begin{equation*}
\lambda_{i}\vert a_{i}\vert^{5}\simeq \lambda_{i}d_{g_{0}}^{5}(a_{i},a_{j})\lesssim \frac{1}{\lambda_{i}^{4}}.
\end{equation*}
However $\vartheta_{i}=0$ on $\{\lambda_{i}\vert a_{i}\vert^{5}\leq \varepsilon\}$ and we conclude 
\begin{equation*}
-\vartheta_{i}\lambda_{i}\partial_{\lambda_{i}}\varepsilon_{i,j}\geq \vartheta_{i}\frac{n-2}{4}\varepsilon_{i,j}.
\end{equation*}
This show the first statement of \eqref{interactions_large_perturbed_case}.  We then compute
\begin{equation*}\begin{split}
- & \sum_{i\neq j} C^{i} \vartheta_{i} \frac{\alpha_{j}}{\alpha_{i}}\lambda_{i}\partial_{\lambda_{i}}\varepsilon_{i,j} 
\\
& = 
-
\sum_{i>j}
[
C^{i}\vartheta_{i}\frac{\alpha_{j}}{\alpha_{i}}
-
C^{j}\vartheta_{j}\frac{\alpha_{i}}{\alpha_{j}}
]
\lambda_{i}\partial_{\lambda_{i}}\varepsilon_{i,j} 
-
\sum_{i<j}C^{i}\vartheta_{i}\frac{\alpha_{j}}{\alpha_{i}}
[
\lambda_{i}\partial_{\lambda_{i}}\varepsilon_{i,j}
+
\lambda_{j}\partial_{\lambda_{j}}\varepsilon_{i,j}
]
\end{split}\end{equation*}
and observe, that the latter sum is non positive, whence
\begin{equation*}\begin{split}
-  \sum_{i\neq j} C^{i} \vartheta_{i} \frac{\alpha_{j}}{\alpha_{i}}\lambda_{i}\partial_{\lambda_{i}}\varepsilon_{i,j} 
\geq &
-
\sum_{i>j}
[
C^{i}\frac{\alpha_{j}}{\alpha_{i}}
-
C^{j}\frac{\vartheta_{j}}{\vartheta_{i}}\frac{\alpha_{i}}{\alpha_{j}}
]
(-\vartheta_{i}\lambda_{i}\partial_{\lambda_{i}}\varepsilon_{i,j}).
\end{split}\end{equation*}
Hence the statement follows for $C\gg 1$ sufficiently large, provided we may uniformly bound
$\vartheta_{j}\lesssim \vartheta_{i}$ for $i>j$, which recalling \eqref{vartheta} translates into 
\begin{equation}\label{vartheta_monotonicity}
\exists\; \kappa\geq 1\;
\forall\;r<s\;:\;\vartheta(r)\leq \kappa \vartheta(s)
\; \text{ for }\; 
\vartheta(t)=\eta(t)+\eta^{\prime}(t)t\ln t,
\end{equation}
i.e. monotonicity in case $\kappa=1$. Recalling furthermore 
\begin{equation*}
\eta\lfloor_{(0,1)}=0,\; \eta\lfloor_{(2,\infty)}=1\; \text{ and }\; \eta^{\prime}\lfloor_{(1,2)}>0, 
\end{equation*}
evidently \eqref{vartheta_monotonicity} is satisfied, whenever $s>1+\delta$ for some $\delta>0$ small, while 
we may assume $\eta^{\prime \prime}\geq 0$ on $(0,1+\delta)$.  Hence $\vartheta$ as a sum of products of non negative monotone functions on $(0,1+\delta)$ is monotone. 
\end{proof}

Together Lemmata \ref{lem_no_mixed_type},\ref{lem_no_concentration_away_from_x_0} and \ref{lem_no_tower_bubbling_at_x_0} show, that a non compact flow line $u$ has to satisfy
$$
 u=\alpha \delta_{a,\lambda}+v \in V(1,\varepsilon)\; \text{ eventually}
 $$

 and 
 $
a\longrightarrow x_{0}=\max_{M}K. 
$

\subsection{Diverging flow lines} \label{sec_diverging_flow_lines}
The only possibility left for a non compact flow line of \eqref{Yamabe_flow} or \eqref{gradient_flow} under Condition \ref{Condition_on_K} is realised. 
\begin{lemma}\label{lem_diverging}
Let $n=5$ and Condition \ref{Condition_on_K} hold true. 
Then for every $\varepsilon>0$ small there exists 
$0<\varepsilon_{0}<\varepsilon$ such, that every
flow line $u$
under \eqref{Yamabe_flow}  or \eqref{gradient_flow} and starting with initial data
\begin{equation*}
u_{0}=\alpha_{0}\varphi_{a_{0}, \lambda_{0}}\in V(1, \varepsilon_{0})
\; \text{ with } \; 
\vert a_{0}\vert <\varepsilon_{0}
\; \text{ and }\; 
\lambda_{0}\vert a_{0}\vert^{2} >\varepsilon^{-1}_{0}
\end{equation*}
remains in $V(1, \varepsilon)$ for all times and 
\begin{align*}
\lambda\longrightarrow  \infty\;\text{ and }\;\vert  a \vert\longrightarrow  0\;\text{ as }\;t\longrightarrow  \infty.
\end{align*} 
\end{lemma}
\begin{proof}
We prove the statement under \eqref{Yamabe_flow}. The proof under \eqref{gradient_flow} is then analogous replacing in particular the appearance of $\vert \delta J\vert$ by $\vert \partial J \vert $. 
In order to prove, that $u$ remains in $V(1, \varepsilon)$ for all times let us define
\begin{equation*}
\begin{split}
T
= 
\sup\{\tau>0\;:\; \;\forall \; 0\leq t <\tau  
\;:\;& 
u\in V(1, \varepsilon), \;
\vert  a \vert< \varepsilon, \;
\lambda\vert  a \vert^{2}> \varepsilon^{-1} 
\}.
\end{split}
\end{equation*} 
We then have to show $T=\infty$. 
We may clearly  assume 
\begin{equation}\label{divergence_integrability}
\int^{\infty}_{0}\vert \delta J(u)\vert^{2}\leq c < \infty .
\end{equation} 
According to Proposition \ref{prop_the_shadow_flow} and using the principal lower bound on $ \partial J$, cf. Definition \ref{def_principally_lower_bounded_of_the_first_variation}, the relevant evolution equations are 
\begin{enumerate}[label=(\roman*)]
 \item \quad 
$
-\frac{\dot{\lambda}}{\lambda}
=
\frac{r}{k}(\gamma_{1}\frac{H(a)}{\lambda^{3}} +\gamma_{2}\frac{\Delta K(a)}{K(a)\lambda^{2}})
(1+o_{\frac{1}{\lambda}}(1))
+
o(\frac{1}{\lambda^{3}})+O(\vert \delta J(u)\vert^{2})
$
 \item \quad  
$
\lambda\dot{a}=\frac{r}{k}
(\gamma_{3}\frac{\nabla K(a)}{K(a)\lambda}  +\gamma_{4}\frac{\nabla \Delta K(a)}{K(a)\lambda^{3}}) (1+o_{\frac{1}{\lambda}}(1))
+
o(\frac{1}{\lambda^{3}})+O(\vert \delta J(u)\vert^{2}),
$ 
 \end{enumerate}
where due to $k=1$ and hence $\frac{r}{k}=J(u)$ we have for some constant $\kappa >0$
during $(0,T)$ 
\begin{equation*}
\frac{r}{k}=\kappa(1+o_{\varepsilon}(1)).
\end{equation*}
Moreover
\begin{equation}\begin{split}\label{derivatives_of_K}
\nabla K(a)=-4\vert  a \vert^{2}a,
\,
\Delta K(a)=-4\cdot 7\vert  a \vert^{2}\; \text{ and }\;\nabla\Delta K(a)=-8\cdot 7 a.
\end{split}\end{equation}
We obtain during $(0,T)$ the simplified evolution equations

\begin{enumerate}[label=(\roman*)]
 \Item \quad 
 \begin{fleqn}[\parindent] \vspace{-2pt}
 \begin{equation}\begin{split}\label{lambda_dot_simplified}
 \textstyle \quad 
-
\frac{\dot{\lambda}}{\lambda}=\kappa\gamma_{2}\frac{\Delta K(a)}{\lambda^{2}}(1+o_{\varepsilon}(1))
+ 
O(\vert \delta J(u)\vert^{2})
\end{split}\end{equation}
\end{fleqn}  
\item \quad 
$
\lambda\dot{a}=\kappa \gamma_{3}\frac{\nabla K(a)}{\lambda}(1+o_{\varepsilon}(1))+O(\vert \delta J(u)\vert^{2}).
$
\end{enumerate}
First note, that during $(0,T)$
\begin{equation}\label{Norm_a_divergence} 
\begin{split}
\partial_{t}\vert  a \vert^{2} 
= & 
\frac{2}{\lambda}\langle a, \lambda\dot{a}\rangle 
=   
2\kappa \gamma_{3}\frac{\langle\nabla K(a),a\rangle}{\lambda^{2}}(1+o_{\varepsilon}(1))
+
O(\frac{\vert  a \vert\vert \delta J(u)\vert^{2}}{\lambda}),
\end{split}
\end{equation} 
whence
$
\partial_{t}\ln \vert  a \vert^{2}\leq O(\frac{\vert \delta J(u)\vert^{2}}{\lambda\vert  a \vert}).
$
But during $(0,T)$ by definition
$$\lambda \vert  a \vert =\lambda^{\frac{1}{2}}(\lambda \vert  a \vert^{2})^{\frac{1}{2}}>c\varepsilon^{-1},$$
whence  $\vert  a \vert $ remains uniformly small, e.g.$\vert  a \vert\leq C\varepsilon_{0}$. 
Secondly
\begin{equation*}
\begin{split}
(\lambda\Delta K(a))' 
= & 
\frac{\dot{\lambda}}{\lambda}\lambda\Delta K(a)+\langle\nabla\Delta K(a), \lambda\dot{a}\rangle\\
= & 
-\kappa
\gamma_{2}\frac{\vert\Delta K(a)\vert^{2}}{\lambda}(1+o_{\varepsilon}(1))
 +
\kappa\gamma_{3}\frac{\langle\nabla\Delta K(a), \nabla K(a)\rangle}{\lambda}(1+o_{\varepsilon}(1))\\
  & +O((\vert\lambda\Delta K(a)\vert+\vert\nabla\Delta K(a)\vert)\vert \delta J(u)\vert^{2}),
\end{split}
\end{equation*} 
and hence, since $\vert \lambda \Delta K(a)\vert=4\cdot 7 \lambda \vert  a \vert^{2}\geq c\varepsilon^{-1}$ 
and 
$\vert \nabla \Delta K(a)\vert \leq C\varepsilon$
during $(0,T)$, 
\begin{equation*}
\begin{split}
\frac{(\lambda\Delta K(a))' }{\kappa} 
= & 
(
-
4^{2}\cdot 7^{2}\gamma_{2}\vert  a \vert ^{4}
+
4\cdot 8 \cdot 7\gamma_{3}\vert  a \vert^{4}
)
\frac{1+o_{\varepsilon}(1)}{\lambda}
\end{split}
\end{equation*} 
up to some 
$$O(\vert\lambda\Delta K(a)\vert\vert \delta J(u)\vert^{2}).$$
Due to  $\frac{\gamma_{3}}{\gamma_{2}}=3$, cf.  the proof of Proposition  \ref{I-prop_n=5} in \cite{may-cv}, this shows
$$
(\lambda\Delta K(a))' 
\leq 
O(\vert\lambda\Delta K(a)\vert\vert \delta J(u)\vert^{2})
$$
and therefore
$
\partial_{t}\ln(-\lambda\Delta K(a))\geq O(\vert \delta J(u)\vert^{2}).
$
We conclude using \eqref{divergence_integrability}, that 
\begin{equation*}\begin{split}    
4\cdot 7 \lambda \vert  a \vert^{2} 
=
-\lambda\Delta K(a)
\geq 
4\cdot 7\lambda_{0}\vert a_{0}\vert^{2}e^{-C\int_{0}^{\infty}\vert \delta J(u)\vert^{2}}
\end{split}\end{equation*}
remains during $(0,T)$ uniformly large, say $\lambda \vert  a \vert^{2}\geq c\varepsilon_{0}^{-1}$.
As a consequence
\begin{equation*}
\begin{split}
-\frac{\dot{\lambda}}{\lambda} 
= & 
\kappa \gamma_{2}\frac{\Delta K(a)}{\lambda^{2}}(1+o_{\varepsilon}(1))
= 
-4\cdot 7 \kappa \gamma_{2}\frac{\vert  a \vert^{2}}{\lambda^{2}}(1+o_{\varepsilon}(1))
\leq 
-\frac{4\cdot 7 \kappa \gamma_{2}c}{\varepsilon_{0}\lambda^{3}}
\end{split}
\end{equation*} 
up to some $O(\vert \delta J(u)\vert^{2})$, 
whence
$$
\partial_{t}\lambda^{3}+\lambda^{3}O(\vert \delta J(u)\vert^{2})
\geq
\frac{4\cdot 7 \kappa \gamma_{2}c}{3\varepsilon_{0}}
=
C_{0}.
$$
Letting $\vartheta=\lambda^{3}$ this becomes 
$
\dot{\vartheta}+\vartheta O(\vert \delta J(u)\vert^{2})
\geq
C_{0}
.
$
Thus 
there holds
\begin{equation*}
\begin{split}
\dot{\tau}(t) 
= & (\dot{\vartheta}+\vartheta O(\vert \delta J(u)\vert^{2}))(t)
e^{\int_{0}^{t}O(\vert \delta J(u)\vert^{2})}
\geq 
C_{0}
e^{\int_{0}^{t}O(\vert \delta J(u)\vert^{2})}
\end{split}
\end{equation*} 
for $\tau(t)=\vartheta(t)e^{\int_{0}^{t}O(\vert \delta J(u)\vert^{2})}$
and therefore
\begin{equation}\label{tau_evolution_divergence}
\begin{split}
\dot{\tau}(t) 
\geq C_{0}e^{-C\int_{0}^{\infty}\vert \delta J(u)\vert^{2}}
\end{split}
\end{equation}
whence
\begin{equation}\begin{split}\label{theta_estimate_divergence}
\vartheta(0)=\tau(0)\leq\tau(t)=\vartheta(t)e^{\int_{0}^{t}O(\vert \delta J(u)\vert^{2})}\leq
\vartheta(t)e^{C\int^{\infty}_{0}\vert \partial J(u)\vert^{2}},
\end{split}\end{equation}
so $\vartheta$ and thus $\lambda$ remain uniformly large, say $\lambda \geq c\varepsilon_{0}^{-1}$.
In summa we cannot escape from
\begin{equation}\label{in_summa}
\begin{split}
\vert  a \vert <C\varepsilon_{0},\;  \lambda \vert  a \vert^{2}>c\varepsilon_{0}^{-1}
\; \text{ and }\; \lambda >c\varepsilon_{0}^{-1}
\end{split}
\end{equation} 
during $(0,T)$. 
Therefore $T=\infty$ follows, if and as we shall prove
\begin{equation*}
u\in V(1, \frac{\varepsilon}{2})\; \text{ during }\; (0,T).
\end{equation*} 
By definition \ref{def_V(omega,p,e)} and the remarks thereafter this is equivalent to showing
\begin{equation*}
\vert 1-\frac{r\alpha^{\frac{4}{n-2}} K(a)}{4n(n-1)k}\vert, \Vert u-\alpha \varphi_{a, \lambda}\Vert = \Vert v \Vert<\frac{\varepsilon}{2}.
\end{equation*} 
To that end let us expand using $k= 1$
\begin{equation*}
\begin{split}
J(u)
= & 
r
=
\int L_{g_{0}}uu
=
\int L_{g_{0}}(\alpha \varphi_{a, \lambda}+v)(\alpha \varphi_{a, \lambda}+v) \\
= &
\alpha^{2}\int L_{g_{0}}\varphi_{a, \lambda}\varphi_{a, \lambda} 
+
2\alpha \int L_{g_{0}}\varphi_{a, \lambda}v
+
\int L_{g_{0}}vv.
\end{split}
\end{equation*} 
Since 
$
L_{g_{0}}\varphi_{a, \lambda}
=
4n(n-1)\varphi_{a,\lambda}^{\frac{n+2}{n-2}}+o_{\frac{1}{\lambda}}(1),
$
we find by simple expansions
\begin{equation}\label{divergence_v_interaction}
\begin{split}
&
\frac{\int L_{g_{0}}\varphi_{a, \lambda}v }{4n(n-1)}
= 
\int \varphi_{a, \lambda}^{\frac{n+2}{n-2}}v
= 
\int K \varphi_{a, \lambda}^{\frac{n+2}{n-2}}v 
= 
\alpha^{-\frac{4}{n-2}}\int K (u-v)^{\frac{4}{n-2}}\varphi_{a, \lambda}v \\
& \quad =  
-\frac{4}{n-2}\alpha^{-\frac{4}{n-2}}\int Ku^{\frac{6-n}{n-2}}\varphi_{a, \lambda} v^{2}
= 
-\frac{4}{n-2}\alpha^{-1}\int K\varphi_{a, \lambda}^{\frac{4}{n-2}} v^{2}
\end{split}
\end{equation} 
 up to some  
$o_{\frac{1}{\lambda}+\vert  a \vert}(1)+o(\Vert v \Vert^{2})$, where we made use of the orthogonality
$$\int Ku^{\frac{4}{n-2}}\varphi_{a,\lambda}v=0$$
considered under \eqref{Yamabe_flow}.  Hence and still up to some $o_{\frac{1}{\lambda}+\vert  a \vert}(1)+o(\Vert v \Vert^{2})$ 
\begin{equation*}
\begin{split}
J(u)
= & 
4n(n-1)c_{1}\alpha^{2} 
+
\int L_{g_{0}}vv
-
\frac{32n(n-1)}{n-2}\int \varphi_{a, \lambda}^{\frac{4}{n-2}}v^{2},
\end{split}
\end{equation*}   
cf. Lemma \ref{lem_interactions}. 
On the other hand we have up to some $o(\Vert v \Vert^{2})$
\begin{equation*}
\begin{split}
1
= &
\int Ku^{\frac{2n}{n-2}}
=
\int K(\alpha \varphi_{a, \lambda}+v)^{\frac{2n}{n-2}}\\
= &
\alpha^{\frac{2n}{n-2}}\int K\varphi_{a, \lambda}^{\frac{2n}{n-2}}
+
\frac{2n}{n-2}\alpha^{\frac{n+2}{n-2}}\int K\varphi_{a, \lambda}^{\frac{n+2}{n-2}}v 
+
\frac{n}{n-2}\frac{n+2}{n-2}\alpha^{\frac{4}{n-2}}\int K\varphi_{a, \lambda}^{\frac{4}{n-2}}v^{2}.
\end{split}  
\end{equation*}
Considering the second summand above we obtain
using  \eqref{divergence_v_interaction} 
\begin{equation*}
\begin{split}
1
= &
\alpha^{\frac{2n}{n-2}}c_{1}
+
\frac{n(n-6)}{(n-2)^{2}}
\alpha^{\frac{4}{n-2}}\int \varphi_{a, \lambda}^{\frac{4}{n-2}}v^{2}
+
o_{\frac{1}{\lambda}+\vert  a \vert}(1)+o(\Vert v \Vert^{2}),
\end{split}
\end{equation*}
whence
\begin{equation*}
\alpha
=
c_{1}^{-\frac{n-2}{2n}}
+
\frac{6-n}{2(n-2)}c_{1}^{-\frac{n+2}{2n}}\int \varphi_{a, \lambda}^{\frac{4}{n-2}}v^{2}
+
o_{\frac{1}{\lambda}+\vert  a \vert}(1)+o(\Vert v \Vert^{2})
\end{equation*} 
and therefore
\begin{equation*}
c_{1}\alpha^{2}
=
c_{1}^{\frac{2}{n}}
+
\frac{6-n}{n-2}\int \varphi_{a, \lambda}^{\frac{4}{n-2}}v^{2}
+
o_{\frac{1}{\lambda}+\vert  a \vert}(1)+o(\Vert v \Vert^{2}).
\end{equation*} 
Consequently and up to some $o_{\frac{1}{\lambda}+\vert  a \vert}(1)+o(\Vert v \Vert^{2})$
\begin{equation*}
\begin{split}
J(u)
= &
4n(n-1)c_{1}^{\frac{2}{n}} 
+
\int L_{g_{0}}vv
-
4n(n-1)\frac{n+2}{n-2}
\int \varphi_{a, \lambda}^{\frac{4}{n-2}}v^{2} 
\end{split}
\end{equation*} 
and, since the latter quadratic form in $v$ corresponding to $\partial^{2}J(\varphi_{a,\lambda})v^{2}$ is well known to be positive, we obtain with some uniform $c>0$ 
\begin{equation*}
\begin{split}
J(u)  
\geq 
4n(n-1)c_{1}^{\frac{2}{n}} 
+
o_{\frac{1}{\lambda}+\vert  a \vert}(1)+c\Vert v \Vert^{2}
.
\end{split}
\end{equation*} 
But $J(u)\leq J(u_{0})=4n(n-1)c_{1}^{\frac{2}{n}}+o_{\frac{1}{\lambda_{0}}+\vert a_{0}\vert}(1)$
as $u_{0}=\alpha_{0}\varphi_{a_{0},\lambda_{0}}$
and therefore
\begin{equation}\label{v_control_diverging_scenario}
\Vert v \Vert^{2}=o_{\frac{1}{\lambda}+\frac{1}{\lambda_{0}}+\vert  a \vert + \vert a_{0}\vert}(1)
\end{equation}  
remains uniformly small during $(0,T)$, cf. \eqref{in_summa}. 
Finally note, that 
\begin{equation*}  
\begin{split}
\frac{r\alpha^{\frac{4}{n-2}}K(a)}{4n(n-1)k}
=
\frac{K(a)\int L_{g_{0}}\varphi_{a,\lambda}\varphi_{a,\lambda}}{4n(n-1)\int K\varphi_{a,\lambda}^{\frac{2n}{n-2}}}
+
o_{\Vert v \Vert}(1) 
=
1
+
o_{\frac{1}{\lambda}+\Vert v \Vert}(1),
\end{split}
\end{equation*} 
whence  by virtue of  \eqref{v_control_diverging_scenario} 
\begin{equation*}  
\vert 1-\frac{r\alpha^{\frac{4}{n-2}}K(a)}{4n(n-1)k}\vert
=
o_{\frac{1}{\lambda}+\frac{1}{\lambda_{0}}+\vert  a \vert + \vert a_{0}\vert}(1)
\end{equation*}
and therefore remains uniformly small, cf. \eqref{in_summa}. 
This completes the proof of $T=\infty$. Then  by \eqref{tau_evolution_divergence}
$
\tau\geq ct,      
$
 whence    $\vartheta=\lambda^{3}\geq \tilde ct$
according to \eqref{theta_estimate_divergence}.  This shows $\lambda \longrightarrow \infty$.
Finally by \eqref{derivatives_of_K} and \eqref{Norm_a_divergence} 
\begin{equation*}
\begin{split}
\partial_{t} \vert  a \vert^{2} \leq
-c\frac{ \vert  a \vert^{4}}{\lambda^{2}}  
+
O(\frac{ \vert  a \vert\vert \delta J(u)\vert^{2}}{\lambda})
=
c
\vert  a \vert^{2}
(
-
\frac{\vert  a \vert^{2}}{\lambda^{2}}
+
O(\frac{ \vert \delta J(u)\vert^{2}}{\vert  a \vert\lambda})
)
\; \text{ for some }\; c>0. 
\end{split}
\end{equation*}
Since $\lambda \vert  a \vert^{2}$ and therefore $\lambda \vert  a \vert$ as well remain large, cf. \eqref{in_summa}, we obtain
\begin{equation*}
\partial_{t}\ln \vert  a \vert^{2}
\leq 
-c\frac{\vert  a \vert^{2}}{\lambda^{2}}+O( \vert \delta J(u)\vert^{2}),
\end{equation*} 
whence due to \eqref{derivatives_of_K} and \ref{lambda_dot_simplified} for some $\tilde c>0$
\begin{equation*}
\partial_{t}\ln \vert  a \vert^{2}
\leq 
-\tilde c\frac{\dot \lambda}{\lambda}+O( \vert \delta J(u)\vert^{2})
=
- 
\partial_{t}\ln \lambda^{\tilde c}+O(\vert \delta J(u)\vert^{2}).
\end{equation*} 
Therefore $\lambda\longrightarrow  \infty$ implies $\vert  a \vert\longrightarrow  0$.
\end{proof} 

\subsection{Modifying the gradient flow}  
\label{sec_modifying the gradient flow}
We finally  discuss how to compactify \eqref{Yamabe_flow} and \eqref{gradient_flow} in the situation of Lemma \ref{lem_diverging}.   From Section \ref{sec_diverging_flow_lines} the only critical value for a non compact flow line is
\begin{equation*}
J_{\infty}=J(\varphi_{x_{0},\infty})=\frac{c_{0}}{K^{\frac{n-2}{n}}(x_{0})},\; c_{0}>0.
\end{equation*}
Hence it is sufficient to only modify \eqref{Yamabe_flow} and \eqref{gradient_flow} on
\begin{equation*}
\mathcal{M}_{\delta}=\{J_{\infty}-\delta<J<J_{\infty}+\delta\},\; 0<\delta \ll 1.
\end{equation*}
We then pass from \eqref{Yamabe_flow} to \eqref{gradient_flow} on $\mathcal{M}_{\delta}$ and are left with suitably compactifying \eqref{gradient_flow} on $\mathcal{M}_{\delta}$. Clearly we may restrict ourselves to modifications on
\begin{equation*}
\mathcal{N}_{a,\varepsilon}
=
V(1,\varepsilon)\cap \{d(a,x_{0})<\varepsilon\} \subset \mathcal{M}_{\delta}
\end{equation*}
for sufficiently small $0<\varepsilon \ll \delta$.  To that end consider a cut-off function
\begin{equation*}
 \eta_{1}\in C^{\infty}(\R_{+},[0,1])\; \text{ with }\; 
 \eta_{1} \lfloor_{(0,1)}=1,\; \eta_{1}\lfloor_{(2,\infty)}=0\; \text{ and }\; \eta_{1}^{\prime}\leq 0
\end{equation*}
and let for $0<\epsilon \ll \varepsilon$
\begin{equation*}
\eta_{V}=\eta(\frac{d(\cdot,V(1,\frac{\varepsilon}{2}))}{\epsilon}) \; \text{ on }\;  X
\; \text{ and }\; 
\eta_{a}=\eta_{1}(\frac{\vert  a \vert}{\epsilon}) \; \text{ on  }\; V(1,\varepsilon),
\end{equation*}
where $\vert \cdot \vert$ denotes the euclidean distance from $x_{0}$ in conformal normal coordinates around $x_{0}$.  
Moreover consider a second cut-off function 
\begin{equation*}
  \eta_{2}\in C^{\infty}(\R_{+},[0,1])\; \text{ with }\; 
 \eta_{2} \lfloor_{(0,1)}=0,\; \eta_{2}\lfloor_{(2,\infty)}=1\; \text{ and }\; \eta^{\prime}_{2}\geq 0
\end{equation*}
and let
\begin{equation*}
\eta_{a,\lambda}=\eta_{2}(\frac{\lambda \vert  a \vert^{2}}{\epsilon})
\; \text{ on }\; \mathcal{N}_{a,\epsilon}.  
\end{equation*}
Hence $\eta_{V}\eta_{a}\eta_{a,\lambda}$ is well defined on $X$ and 
\begin{equation*}
supp(\eta_{V}\eta_{a}\eta_{a,\lambda})\subset supp(\eta_{V}\eta_{a})\subset \mathcal{N}_{a,\varepsilon}\subset \mathcal{M}_{\delta}. 
\end{equation*}
We then consider for some $C\geq 1$
\begin{equation}\label{W_definition}
W=-\varepsilon\eta_{V}\eta_{a}\eta_{a,\lambda} 
(
\frac{\alpha^{-1}\nabla K(a)}{ \vert \nabla K(a)\vert \lambda}\frac{\nabla_{a}}{\lambda}\varphi_{a,\lambda}
-
C\frac{v}{\lambda} 
)
\end{equation}
as a bounded, locally Lipschitz vectorfield on $X$, which is well defined due to 
\begin{equation*}
\nabla K(a)=-4\vert  a \vert^{2} a \neq 0 \; \text{ on }\; supp(\eta_{a,\lambda}),
\end{equation*}
and study the flow generated by 
\begin{equation}\label{gradient_flow_modified_discussion}
\partial_{t}u=-\frac{r}{2k}(\nabla J(u)+W+\frac{\int Ku^{\frac{n+2}{n-2}}(\nabla J(u)+W)}{k}u).
\end{equation}
Clearly $k=1$ is preserved as is positivity $u>0$ along flow lines
and consequently \eqref{gradient_flow_modified_discussion} induces a flow on $X$. 
Indeed 
\begin{equation*}
-W
\geq 
\varepsilon\eta_{V}\eta_{a}\eta_{a,\lambda}(-c\frac{\varphi_{a,\lambda}}{\lambda}-C\frac{v}{\lambda})
\geq 
-\frac{C\varepsilon}{\lambda}\eta_{V}\eta_{a}\eta_{a,\lambda}u \geq -u
\end{equation*}
for $C\gg 1$ sufficiently large, whence we obtain in combination with \eqref{gradient_positive}
\begin{equation*}
\partial_{t}u\geq -\tilde C(1+\vert \partial J(u)\vert)u
\end{equation*}
and therefore $u$ exists positively for all times, provided we have uniform a priori bounds on $\vert \partial J(u)\vert$, 
which we derive from Proposition \ref{prop_derivatives_of_J} using $k=1$ and the boundedness of energy along a flow line. 
The latter boundedness follows from the subsequent  Lemma \ref{lem_modified_flow_energy_consumption}. 

\begin{lemma}\label{lem_modified_flow_energy_consumption}
Along a flow line there holds 
$
\partial_{t}J(u) \lesssim -\vert \partial J(u)\vert^{2}. 
$
\end{lemma}
\begin{proof}
Since $\partial J(u)u=0$ by scaling invariance, we clearly have
\begin{equation*}
\partial_{t}J(u) 
=
-\frac{r}{2k}
(
\Vert \nabla J(u)\Vert^{2}
+
\partial J(u)W 
).
\end{equation*}
Then Proposition \ref{prop_a-priori_estimate_on_v} and the principal lower bound on $\partial J$ yield
\begin{equation*}
 \partial J(u)\frac{v}{\lambda}
 =
 O(\frac{\vert \partial J(u)\vert^{2}}{\lambda}), 
\end{equation*}
cf. Definition 
\ref{def_principally_lower_bounded_of_the_first_variation}, 
whence 
\begin{equation*}
\partial_{t}J(u) 
=
-\frac{r}{2k}
(
\vert \partial J(u)\vert^{2}(1+o(1))
-
\varepsilon\eta_{V}\eta_{a}\eta_{a,\lambda}\frac{\alpha^{-1}\nabla K(a)}{ \vert \nabla K(a)\vert \lambda}\partial J(u)\frac{\nabla_{a}}{\lambda}\varphi_{a,\lambda} 
).
\end{equation*}
From Proposition \ref{prop_simplifying_ski} and \eqref{small_difference} we then find     
\begin{equation*}
\partial J(u)\frac{\nabla_{a}}{\lambda}\varphi_{a,\lambda}
=
-\sigma_{3,\cdot}
=
-
4n(n-1) e_{3}
\alpha
\frac{\nabla K(a)}{K(a)\lambda} 
+
o_{\varepsilon}
(
\frac{1}{\lambda^{3}}
)
+
O
(
\vert \partial  J(u)\vert^{2}
)  
\end{equation*}
using again Proposition \ref{prop_a-priori_estimate_on_v} and the principal lower bound on $\partial J$. Therefore
\begin{equation*}
\partial_{t}J(u) 
=
-\frac{r}{2k}
(
\vert \partial J(u)\vert^{2}(1+o(1))
+
4n(n-1) e_{3}\varepsilon
\eta_{V}\eta_{a}\eta_{a,\lambda}
(
\frac{\vert \nabla K(a)\vert }{K(a)\lambda^{2}} 
+
o_{\varepsilon}(\frac{1}{\lambda^{4}}) 
)).
\end{equation*}
Note, that on $ supp (\eta_{a,\lambda})$ we have $\lambda \vert  a \vert^{2}\geq \varepsilon$ and hence 
$
\frac{\vert \nabla K(a)\vert}{\lambda^{2}}
 \gg \frac{1}{\lambda^{4}}.
$
\end{proof}
In particular the flow generated by  \eqref{gradient_flow_modified_discussion} decreases energy and we have 
\begin{equation*}
 \int^{\infty}_{0}\vert \partial J(u)\vert^{2}<\infty
\end{equation*}
just like under 
\eqref{gradient_flow}. 
Since \eqref{gradient_flow_modified_discussion} coincides with \eqref{gradient_flow} outside $V(1,\varepsilon)$, whereupon the flow generated by \eqref{gradient_flow} is compact, cf. Section \ref{sec_compact_regions}, every non compact flow line $u$ for \eqref{gradient_flow_modified_discussion} has to enter $V(1,\varepsilon)$ for at least a sequence in time. If we suppose, that $u$ does not remain in $V(1,2\varepsilon)$ eventually, then there exists
\begin{equation*}
s_{1}\leq s_{1}^{\prime}\leq \ldots \leq s_{k} \leq s_{k}^{\prime} \leq \ldots
\; \text{ with }\; s_{k},s_{k}^{\prime}\xrightarrow{k\to \infty} \infty
\end{equation*}
such, that 
\begin{equation*}
u_{s_{k}}\in \partial V(1,\varepsilon),\; u_{s_{k}^{\prime}}\in \partial V(1,2\varepsilon)
\; \text{ and }\; 
u\in V(1,2\varepsilon)\setminus V(1,\varepsilon) \; \text{ during }\; (s_{k},s_{k}^{\prime}).  
\end{equation*}
However, since $\Vert \partial_{t}u\Vert \leq C$ under $\eqref{gradient_flow_modified_discussion}$, as
$$\Vert \nabla J(u)\Vert=\vert \partial J(u)\vert$$ is uniformly bounded along a flow line,
and 
$$d(\partial V(1,2\varepsilon),\partial V(1,\varepsilon))\geq\tilde \varepsilon,$$
we find $\vert s_{k}^{\prime}-s_{k}\vert \geq \frac{\tilde \varepsilon}{C}$. Moreover there holds 
\begin{equation*} 
\vert \partial J \vert \geq \bar \varepsilon \; \text{ on } \; V(1,2\varepsilon)\setminus V(1,\varepsilon)
\end{equation*}
by combining Proposition \ref{prop_a-priori_estimate_on_v} and (i) from Proposition \ref{prop_analysing_ski} with the principal lower bound on $\partial J $, cf. Definition \ref{def_principally_lower_bounded_of_the_first_variation}.  Therefore we infer from Lemma \ref{lem_modified_flow_energy_consumption}
\begin{equation*}
J(u_{s_{k}^{\prime}}) - J(u_{s_{k}})
= 
\int^{s_{k}^{\prime}}_{s_{k}}\partial_{t}J(u) 
\leq 
-
c\int^{s_{k}^{\prime}}_{s_{k}}\vert \partial J(u)\vert^{2} 
\leq 
-\frac{c\bar \varepsilon^{2}\tilde \varepsilon}{C}
\end{equation*}
and hence iteratively
\begin{equation*}
\begin{split}
J(u_{s_{k}^{\prime}})
= &
J(u_{s_{k}^{\prime}})-J(u_{s_{k}})+J(u_{s_{k}})
\leq 
J(u_{s_{k}^{\prime}})-J(u_{s_{k}})+J(u_{s_{k-1}^{\prime}})
\\ 
\leq &
J(u_{s_{1}})+\sum_{i=1}^{k}\left(J(u_{s_{k}^{\prime}})-J(u_{s_{k}})\right),
\end{split}
\end{equation*}
which necessitates $J(u_{s_{k}})\longrightarrow -\infty$, a contradiction.   
Hence we may assume 
$$u\in V(1,2\varepsilon)\; \text{ eventually}.$$ 
On the other hand, since by Lemma \ref{lem_modified_flow_energy_consumption} every flow line up to a sequence in time is a Palais-Smale, cf. \eqref{flow_palais_smale_gradient},  we may assume, that $u$ is precompact in some $V(\omega,p,\delta)$ for  every $\delta>0$.  Since 
\begin{equation*}
d(V(\omega,p,\delta),V(1,2\delta))>\tilde \delta
\; \text{ in case }\; \omega \neq 0 \; \text{ or }\; p\neq 1
\end{equation*} 
for all $\delta>0$ sufficiently small, the same energy consumption argument as before would lead to the same contradiction. Hence necessarily 
\begin{equation}\label{concentration_compactness_for_modified_flow}
u=\alpha \varphi_{a,\lambda} + v \in V(1,\delta) \; \text{ for every } \; \delta>0 \; \text{ eventually.}
\end{equation} 
In particular we may assume $\eta_{V}=1$ eventually for a non compact flow line. 


\bigskip

So let us analyse the impact on  the shadow flow,  when   passing from \eqref{gradient_flow} to \eqref{gradient_flow_modified_discussion}, in particular on the evolution equations for $a$ and $\lambda$. Comparing to
Section \ref{sec_shadow_flow} we find  in the present one bubble scenario 
\begin{enumerate}[label=(\roman*)]
 \item \quad 
 $
 \dot \xi_{k}=(\frac{\dot \alpha}{\alpha},-\frac{\dot \lambda}{\lambda},\lambda \dot a)
 \; \text{ and } \; 
 \phi_{l}=(\varphi_{a,\lambda},-\lambda \partial_{\lambda}\varphi_{a,\lambda},\frac{\nabla_{a}}{\lambda}\varphi_{a,\lambda})
 $
 \item \quad 
 $\Xi_{k,l}
=
4n(n-1)\alpha c_{k}\delta_{k,l}
+
 O
 ( 
 \frac{1}{\lambda^{2}}
 +  
 \vert \partial J(u)\vert 
 )$
 \item \quad
 $
 \Xi_{k,l}\dot \xi^{k}
 =
 \langle \partial_{t}u,\phi_{l}\rangle. 
 $
\end{enumerate}
To achieve  the simple form of $\Xi$ in (ii) above, we applied    Proposition \ref{prop_a-priori_estimate_on_v}
and the principal lower bound on $ \partial J $, cf. Definition \ref{def_principally_lower_bounded_of_the_first_variation}, to \eqref{Xi}. Note, that
due to $k=1$, cf. Proposition \ref{prop_derivatives_of_J}, 
\begin{equation*}
\begin{split}
\frac{r}{k}\int Ku^{\frac{n+2}{n-2}}
&
(\nabla J(u)+W)
=
-
\partial J(u)(\nabla J(u)+W)
+
\int L_{g_{0}}u(\nabla J(u)+W)
 \\
=&
-\vert \partial J(u)\vert^{2} + \partial J(u)u
-
\partial J(u)W+\int L_{g_{0}}uW ,
\end{split}
\end{equation*}
where $\partial J(u)u=0$ by scaling invariance, 
$\partial J(u)W=O(\frac{\vert \partial J(u)\vert}{\lambda})$ by \eqref{W_definition} and 
\begin{equation*}
\begin{split}
\int L_{g_{0}}uW
= &
-\varepsilon\eta_{V}\eta_{a}\eta_{a,\lambda} 
(
\frac{\alpha^{-1}\nabla K(a)}{ \vert \nabla K(a)\vert \lambda}\int L_{g_{0}}u\frac{\nabla_{a}}{\lambda}\varphi_{a,\lambda}
-
C\int L_{g_{0}}u\frac{v}{\lambda} 
) \\
= &
O(\frac{1}{\lambda^{2}}+\Vert v \Vert^{2})
\end{split}
\end{equation*}
by orthogonalities $\langle v,\phi_{k}\rangle=0$ and 
$\int L_{g_{0}}\varphi_{a,\lambda}\frac{\nabla_{a}}{\lambda}\varphi_{a,\lambda}=O(\frac{1}{\lambda^{2}})$. 
Hence 
\begin{equation*}
 \int Ku^{\frac{n+2}{n-2}}(\nabla J(u)+W)
 =
 O (\frac{1}{\lambda^{2}}+\vert \partial J(u)\vert^{2})
\end{equation*}
absorbing $\Vert v \Vert^{2}$ by Proposition \ref{prop_a-priori_estimate_on_v} and the principal lower bound on 
$
\partial J.
$
We therefore have for \eqref{gradient_flow_modified_discussion}, cf. Proposition \ref{prop_simplifying_ski},
\begin{equation*}
\begin{split}
\langle \partial_{t}u,\phi_{l}\rangle 
= &
-\frac{r}{2k}
(
\langle 
\nabla J(u)
,
\phi_{l}\rangle_{L_{g_{0}}}
+
\langle  
W
,
\phi_{l}\rangle_{L_{g_{0}}}
+
O (\frac{1}{\lambda^{2}}+\vert \partial J(u)\vert^{2})
\langle 
u 
,
\phi_{l}\rangle_{L_{g_{0}}}
) \\
= &
\frac{r}{k}\sigma_{l}  
+
\varepsilon \eta_{V}\eta_{a}\eta_{a,\lambda}\frac{r}{2k} 
\frac{\alpha^{-1}\nabla K(a)}{ \vert \nabla K(a)\vert \lambda}
\langle \frac{\nabla_{a}}{\lambda}\varphi_{a,\lambda},\phi_{l}\rangle_{L_{g_{0}}} \\
& +
O (\frac{1}{\lambda^{2}}+\vert \partial J(u)\vert^{2})
\frac{r}{k}
\langle 
\varphi_{a,\lambda} 
,
\phi_{l}\rangle_{L_{g_{0}}}
\end{split}
\end{equation*}
and obtain using $\int L_{g_{0}}\phi_{k}\phi_{l}=c_{k}\delta_{k,l}+O(\frac{1}{\lambda^{2}})$
\begin{equation*}
\begin{pmatrix}
\langle \partial_{t}u,\phi_{1}\rangle \\
\langle \partial_{t}u,\phi_{2}\rangle \\
\langle \partial_{t}u,\phi_{3}\rangle 
\end{pmatrix}
=
\frac{r}{k}
\begin{pmatrix}
\sigma_{1} +O(\frac{1}{\lambda^{2}})\\
\sigma_{2} \\
\sigma_{3} 
+
\frac{\varepsilon}{2} c_{3}\eta_{V}\eta_{a}\eta_{a,\lambda}\frac{\alpha^{-1}}{\lambda}\frac{\nabla K(a)}{\vert \nabla K(a)\vert}
\end{pmatrix}
+
o_{\varepsilon}(\frac{1}{\lambda^{3}})
+
O(\vert \partial J(u)\vert^{2})
\end{equation*}
and hence by matrix inversion
\begin{equation*}
\dot \xi_{k}
=
\Xi^{l}_{k}\langle \partial_{t}u,\phi_{l}\rangle
=
\frac{(1+O(\frac{1}{\lambda^{2}}+\vert \partial J(u)\vert )r}{4n(n-1)\alpha k}
\begin{pmatrix}
\frac{\sigma_{1}}{c_{1}} +O(\frac{1}{\lambda^{2}})\\
\frac{\sigma_{2}}{c_{2}} \\
\frac{\sigma_{3} }{c_{3}}
+
\frac{\varepsilon}{2} \eta_{V}\eta_{a}\eta_{a,\lambda}\frac{\alpha^{-1}}{\lambda}\frac{\nabla K(a)}{\vert \nabla K(a)\vert}
\end{pmatrix} 
\end{equation*}
up to some 
$
o_{\varepsilon}(\frac{1}{\lambda^{3}})
+
O(\vert \partial J(u)\vert^{2}).
$
Recalling $\sigma_{k}=O(\vert \partial J(u)\vert)$ we may simplify to
\begin{equation*}
\dot \xi_{k}
=
\frac{(1+o_{\frac{1}{\lambda}}(1))r}{4n(n-1)\alpha k}
\begin{pmatrix}
\frac{\sigma_{1}}{c_{1}} \\
\frac{\sigma_{2}}{c_{2}} \\
\frac{\sigma_{3} }{c_{3}}
+
\frac{\varepsilon}{2} (1+o_{\varepsilon}(1))\eta_{V}\eta_{a}\eta_{a,\lambda}\frac{\alpha^{-1}}{\lambda}\frac{\nabla K(a)}{\vert \nabla K(a)\vert}
\end{pmatrix} 
+
o_{\varepsilon}(\frac{1}{\lambda^{3}})
+
O(\vert \partial J(u)\vert^{2}).
\end{equation*}
From Proposition \ref{prop_simplifying_ski} we thus obtain using \eqref{small_difference}, Proposition \ref{prop_a-priori_estimate_on_v} and the principal lower bound on $ \partial J $

\begin{lemma}\label{lem_modified_shadow_flow}
Along \eqref{gradient_flow_modified_discussion} there holds on $V(1,\varepsilon)$
\begin{enumerate}[label=(\roman*)]
\item \quad
$
-\frac{\dot \lambda}{\lambda}
=
\frac{r}{ k}
(
\frac{d_{2}}{c_{2}}\frac{ H(a)}{\lambda^{n-2}}
+
\frac{e_{2}}{c_{2}}\frac{\lap K(a)}{K(a)\lambda^{2}} 
)
(1+o_{\frac{1}{\lambda}}(1))
$
\item \quad 
$
\lambda \dot a
=
\frac{r}{k}
[
 \frac{e_{3}}{c_{3}}\frac{\nabla K(a)}{K(a) \lambda }
+
\frac{e_{4}}{c_{3}}\frac{\nabla \lap K(a)}{K(a)\lambda^{3}}
] 
(1+o_{\frac{1}{\lambda}}(1))
+
\frac{\varepsilon r}{8n(n-1)\alpha^{2} k}
\eta_{V}\eta_{a}\eta_{a,\lambda}\frac{\nabla K(a)}{ \vert \nabla K(a)\vert \lambda}
(1+o_{\varepsilon}(1))
$
\end{enumerate}
up to some $o_{\varepsilon}(\frac{1}{\lambda^{3}})+O(\vert \partial J(u)\vert^{2})$
and for $d(a,x_{0}) \ll 1$ up to the same error
\begin{enumerate}[label=(\roman*)]
\item \quad
$
-\frac{\dot \lambda}{\lambda}
=
\frac{r}{ k}
(
\gamma_{1}\frac{ H(a)}{\lambda^{3}}
-
\gamma_{2}\frac{\vert  a \vert^{2}}{\lambda^{2}} 
)
(1+o_{\frac{1}{\lambda}+\vert  a \vert}(1))
$
\item \quad 
$
\lambda \dot a
=
-
\frac{r}{k}
(
\gamma_{3}\frac{\vert  a \vert^{2}a}{ \lambda } 
(1+o_{\frac{1}{\lambda}+\vert a \vert}(1))
-
\gamma_{4}\varepsilon 
\eta_{V}\eta_{a}\eta_{a,\lambda}\frac{e_{a}}{\lambda}
(1+o_{\varepsilon}(1))
)
$
\end{enumerate} 
with $e_{a}=\frac{a}{\vert  a \vert},\gamma_{1},\ldots,\gamma_{4}>0$. 
\end{lemma}
Clearly the latter version for $d(a,x_{0}) \ll 1$  follows from  \eqref{derivatives_of_K}. 
Comparing  to Proposition \ref{prop_the_shadow_flow}
we observe, that by passing from 
\eqref{gradient_flow} to \eqref{gradient_flow_modified_discussion} we have simply added the term 
\begin{equation*}
\frac{\varepsilon (1+o_{\varepsilon}(1))r}{8n(n-1)\alpha^{2} k}
\eta_{V}\eta_{a}\eta_{a,\lambda}\frac{\nabla K(a)}{ \vert \nabla K(a)\vert \lambda}
\end{equation*}
to the evolution equation of $a$, hence moving $a$ faster towards $x_{0}$.

\subsection{Excluding diverging flow lines}  
\label{sec_excluding_diverging_flow_lines}
As we had, cf. \eqref{concentration_compactness_for_modified_flow}, the only possibility for a diverging  flow line under \eqref{gradient_flow_modified_discussion} is 
$$u=\alpha \varphi_{a,\lambda}+v \in V(1,\varepsilon) \; \text{ eventually for every }\; \varepsilon>0 $$ 
with corresponding modified shadow flow given by Lemma \ref{lem_modified_shadow_flow}, from which
\begin{equation*}
\begin{split}
(\frac{\ln \frac{1}{\lambda}}{K(a)})^{\prime } 
= &
\frac{\ln \lambda}{K^{2}(a)}\langle\frac{\nabla K(a)}{\lambda},\lambda \dot a\rangle
-
\frac{1}{K(a)}\frac{\dot \lambda}{\lambda} \\
\geq &
c
[
\ln \lambda 
(
\frac{\vert \nabla K(a)\vert^{2}}{ \lambda^{2} } 
+ 
\varepsilon 
\eta_{V}\eta_{a}\eta_{a,\lambda}\frac{\vert \nabla K(a)\vert}{\lambda^{2} }
) 
 +
\frac{ H(a)}{\lambda^{3}}
]
\\
& +
\frac{r}{ k}
\frac{e_{2}}{c_{2}}\frac{\lap K(a)}{K^{2}(a)\lambda^{2}} 
(1+o_{\frac{1}{\lambda}}(1))
+
O(\vert \partial J(u)\vert^{2}),
\end{split}
\end{equation*}
as an easy computation shows. Hence $\lambda \longrightarrow \infty$ necessitates 
$$a \longrightarrow x_{0}=\{\nabla K=0\}\cap \mathcal{M}_{\delta}$$
at least for a sequence in times, while on the other hand 
\begin{equation*}
\begin{split}
\partial_{t}\vert \nabla K(a)\vert^{2}
= &
2\nabla^{2} K(a)\frac{\nabla K(a)}{\lambda}\lambda \dot a \\
= &
\frac{\varepsilon (1+o_{\varepsilon}(1))r}{4n(n-1)c_{3}\alpha^{2} k}
\eta_{V}\eta_{a}\eta_{a,\lambda}\frac{\nabla^{2} K(a) \nabla K(a)\nabla K(a)}{ \vert \nabla K(a)\vert \lambda^{2}} \\
& +
 O(\frac{1}{\lambda^{6}}+\frac{\vert \nabla K(a)\vert^{2}}{\lambda^{2}}+\vert \partial J(u)\vert^{2})
 \leq 
 O(\vert \partial J(u)\vert^{2})
\end{split}
\end{equation*}
due to the principal lower bound on $\partial J$, cf. Definition \ref{def_principally_lower_bounded_of_the_first_variation}, and 
$$\nabla^{2}K\leq 0 \; \text{ close to }\; x_{0}=\{K=\max K\},$$ 
i.e. on $supp (\eta_{a})$.  
Therefore and by $\int^{\infty}_{0}\vert \partial J(u)\vert^{2}<\infty$ we find, that necessarily
\begin{equation*}
\lambda \longrightarrow \infty
\Longrightarrow 
a \longrightarrow x_{0}. 
\end{equation*}
In particular we may assume $\eta_{V}=\eta_{a}=1$ from now on. Then on
\begin{equation*}
\{\eta_{a,\lambda}=1\}=\{\lambda \vert a \vert^{2}\geq 2\varepsilon\} 
\end{equation*}
we find from Lemma \ref{lem_modified_shadow_flow} in its refined version for $a$ close to $x_{0}$
\begin{equation*}  
\begin{split}
(\lambda \vert a &\vert^{2})^{\prime}
= 
\lambda \vert a \vert^{2}\frac{\dot \lambda}{\lambda}
+
2\langle a,\lambda \dot a \rangle \\
\leq  &
-
\lambda \vert a \vert^{2} 
(
\tilde \gamma_{1}\frac{ H(a)}{\lambda^{3}}
-
\tilde \gamma_{2}\frac{\vert  a \vert^{2}}{\lambda^{2}} 
+
O(\vert \partial J(u)\vert^{2}) 
)
-
(
\tilde \gamma_{4}\varepsilon \frac{\vert a \vert}{\lambda}
+
O(\vert a \vert\vert \partial J(u)\vert^{2}))
 \\
 \leq  &
 -
 \tilde \gamma_{4}\varepsilon \frac{\vert a \vert}{\lambda}
+
C\lambda \vert a \vert^{2} 
\vert \partial J(u)\vert^{2}
\leq 
C\lambda \vert a \vert^{2} 
\vert \partial J(u)\vert^{2}.
\end{split}
\end{equation*}
Consequently $\lambda \vert a \vert^{2}$ is bounded and considering
$\psi=\max\{2\varepsilon,\lambda \vert a \vert^{2}\}$
there necessarily holds
$$
\int^{\infty}_{0}\frac{\vert a \vert}{\lambda}\chi_{\{\lambda \vert a \vert^{2}\geq 2\varepsilon\}}<\infty.
$$
But then
\begin{equation*}
\begin{split}
\partial_{t}\ln \lambda
\leq  &
-
\tilde \gamma_{1}\frac{ H(a)}{\lambda^{3}}
+
\tilde \gamma_{2}\frac{\vert  a \vert^{2}}{\lambda^{2}} 
+
O(\vert \partial J(u)\vert^{2}) \\
\leq& 
-
\tilde \gamma_{1}\frac{ H(a)}{\lambda^{3}}
+
\tilde \gamma_{2}\frac{2\varepsilon}{\lambda^{3}}
\chi_{\{\lambda \vert a \vert^{2}< 2\varepsilon\}} 
+
\tilde \gamma_{2}\frac{\vert  a \vert^{2}}{\lambda^{2}}
\chi_{\{\lambda \vert a \vert^{2}\geq 2\varepsilon\}} 
+
O(\vert \partial J(u)\vert^{2})  
\end{split}
\end{equation*}
by Lemma \ref{lem_modified_shadow_flow}
and, since $\vert a \vert \ll 1 \ll \lambda$, we obtain for $\varepsilon>0$
 sufficiently small
 \begin{equation*}
\begin{split}
\partial_{t}\ln \lambda
\leq& 
\tilde \gamma_{2}\frac{\vert  a \vert}{\lambda}
\chi_{\{\lambda \vert a \vert^{2}\geq 2\varepsilon\}} 
+
O(\vert \partial J(u)\vert^{2})  
\end{split}
\end{equation*}
and the right hand side is integrable in time. Hence $\lambda \longrightarrow \infty$ is impossible.

 \section{Appendix}\label{sec_appendix}

We first recall some  testings of the derivative $\partial J(u)$ with $\phi_{k,i}$ from \cite{MM1}, where we had worked with the representation of $u\in V(\omega,p,\varepsilon)$ based on minimising 
\begin{equation*}\begin{split}
\int 
Ku^{\frac{4}{n-2}}
\vert 
u
-
u_{\tilde \alpha, \tilde \beta}
-
\tilde\alpha^{i}\varphi_{\tilde a_{i}, \tilde \lambda_{i}}
\vert^{2}
\end{split}\end{equation*}
leading to the orthogonalities 
$
v \in \langle u_{\alpha,\beta},\phi_{k,i} \rangle^{\perp_{Ku^{\frac{4}{n-2}}}}.
$ 
 By  Lemma \ref{lem_comparing_orthogonalities} we may carry over these testings to
the representation induced by the minimising  
\begin{equation*}\begin{split}
\Vert 
u
-
u_{\tilde \alpha, \tilde \beta}
- 
\tilde\alpha^{i}\varphi_{\tilde a_{i}, \tilde \lambda_{i}}
\Vert_{L_{g_{0}}}^{2}.
\end{split}\end{equation*}
\begin{proposition}\label{prop_analysing_ski}
For $u\in V(p, \varepsilon)$   and 
\begin{equation*}\begin{split}
\sigma_{k,i}=-\int (L_{g_{0}}u-r\K u^{\frac{n+2}{n-2}})\phi_{k,i}, \, i=1, \ldots,p, \,k=1,2,3
\end{split}\end{equation*} we have
with constants $b_{1}, \ldots,e_{4}>0$
\begin{enumerate}[label=(\roman*)]
  \Item \quad  
\begin{equation*}\begin{split}
\sigma_{1,i}
= &
\alpha_{i}
[ 
\frac{r\alpha_{i}^{\frac{4}{n-2}}K_{i}}{k}
-
4n(n-1)
]
\int \varphi_{i}^{\frac{2n}{n-2}}\\
& +
\sum_{i\neq j =1}^{p}\alpha_{j}
[
\frac{r\alpha_{j}^{\frac{4}{n-2}}K_{j}}{k}-4n(n-1)
] 
b_{1}\varepsilon_{i,j}
 \\
& 
+
d_{1}\alpha_{i}\frac{ H_{i}}{\lambda_{i} ^{n-2}}
+
e_{1}\frac{r\alpha_{i}^{\frac{n+2}{n-2}}}{k}  \frac{\lap K_{i}}{\lambda_{i} ^{2}} 
+
b_{1}
\frac{r\alpha_{i}^{\frac{4}{n-2}}K_{i}}{k}
\sum_{i \neq j=1}^{p}\alpha_{j}
\varepsilon_{i,j} 
\end{split}\end{equation*}
  \Item \quad  
\begin{equation*}\begin{split}
\sigma_{2,i}
= &
-\alpha_{i}
[
\frac{r\alpha_{i}^{\frac{4}{n-2}}K_{i}}{k}
-
4n(n-1)
]
\int \varphi_{i}^{\frac{n+2}{n-2}}\lambda_{i}\partial_{\lambda_{i}}\varphi_{i}\\
& 
-
b_{2}\sum_{i\neq j =1}^{p}\alpha_{j}
[
\frac{r\alpha_{j}^{\frac{4}{n-2}}K_{j}}{k}-4n(n-1)
] 
\lambda_{i} \partial_{\lambda_{i}}\varepsilon_{i,j}
+
d_{2}\alpha_{i}\frac{ H_{i}}{\lambda_{i} ^{n-2}} \\
& +
e_{2}\frac{r\alpha_{i}^{\frac{n+2}{n-2}}}{k} \frac{\lap K_{i}}{\lambda_{i}^{2}}  - 
b_{2}
\frac{r\alpha_{i}^{\frac{4}{n-2}}K_{i}}{k}
\sum_{i \neq j=1}^{p}\alpha_{j}
\lambda_{i}\partial_{\lambda_{i}}\varepsilon_{i,j} 
\end{split}\end{equation*}
  \Item \quad  
\begin{equation*}\begin{split}
\sigma_{3,i}
= &
\alpha_{i}
[
\frac{r\alpha_{i}^{\frac{4}{n-2}}K_{i}}{k}
-
4n(n-1)
]
\int \varphi_{i}^{\frac{n+2}{n-2}}\frac{1}{\lambda_{i}}\nabla_{a_{i}}\varphi_{i}\\
& +
b_{3}\sum_{i\neq j =1}^{p}\alpha_{j}
[
\frac{r\alpha_{j}^{\frac{4}{n-2}}K_{j}}{k}-4n(n-1)
] 
\frac{1}{\lambda_{i}} \nabla_{ a _{i}}\varepsilon_{i,j}
 \\
& 
+
\frac{r\alpha_{i}^{\frac{n+2}{n-2}}}{k}
[
e_{3}\frac{\nabla K_{i}}{\lambda_{i}}
+
e_{4}\frac{\nabla \lap K_{i}}{\lambda_{i}^{3}} 
] 
+
b_{3}\frac{r\alpha_{i}^{\frac{4}{n-2}}K_{i}}{k}
\sum_{i \neq j=1}^{p}
\frac{\alpha_{j}}{\lambda_{i}}\nabla_{a_{i}}\varepsilon_{i,j} 
\end{split}\end{equation*}
\end{enumerate}
up to some  
$
o_{\varepsilon}
(
 \lambda_{i}^{2-n} 
+
\sum_{i\neq j=1}^{p}\varepsilon_{i,j}
)
+
O(
\sum_{r\neq s}\frac{\vert  \nabla K_{r}\vert^{2}}{\lambda_{r}^{2}}
+
\frac{1}{\lambda_{r}^{4}}
+
\frac{1}{\lambda_{r}^{2(n-2)}}
+
\varepsilon_{r,s}^{2}
+
\Vert v\Vert^{2})
.                               
$
\end{proposition}
\begin{proof}
This follows from Proposition \ref{I-prop_analysing_ski}  from \cite{may-cv} in case 
$$\langle v,\phi_{k,i} \rangle_{Ku^{\frac{4}{n-2}}}=0.$$ 
In case 
$\langle v,\phi_{k,i} \rangle_{L_{g_{0}}}=0$  we have  from Lemma 
\ref{lem_comparing_orthogonalities}
\begin{equation*}
 \Pi^{\top_{Ku^{\frac{4}{n-2}}}}_{\langle \phi_{k,i}\rangle}
v
=
O(
(
\sum_{r\neq s}\frac{\vert  \nabla K_{r}\vert}{\lambda_{r}}
+
\frac{1}{\lambda_{r}^{2}}
+
\varepsilon_{r,s}
+
\Vert v \Vert)\Vert v\Vert)
\end{equation*}
and  may consequently reduce the latter case to the former one. 
\end{proof}
Likewise we may carry over  Proposition \ref{I-prop_analysing_ski_f} from \cite{may-cv} for the case $u\in V(\omega,p,\varepsilon)$ and $\omega>0$. 
We next analyse the gradient orthogonally. 

\begin{proposition}\label{prop_derivatives_on_H}
Let $u=\alpha^{i}\var_{i}+v\in V(p, \varepsilon)$ and 
$$h_{1},h_{2}\in H= H_{u}(p, \varepsilon).$$ 
Then
\begin{enumerate}[label=(\roman*)]
 \item \quad
$
\vert\partial J( \alpha^{i}\varphi_{i}  )\lfloor_{H}\vert  
= 
O(\sum_{r\neq s} \frac{\vert \nabla K_{r}\vert}{\lambda_{r}} 
+ \frac{\vert \lap K_{r}\vert}{\lambda_{r}^{2}} 
+
\lambda_{r}^{2-n} 
+\varepsilon_{r,s}+\Vert v \Vert^{2}+\vert \partial  J(u)\vert)
$
\item \quad and up to some $o_{\varepsilon}(\Vert h_{1} \Vert\, \Vert h_{2} \Vert)$ we have
\begin{equation*}\begin{split}
\frac{1}{2}\partial^{2}J( \alpha^{i}\varphi_{i}   )h_{1}h_{2}
= &
k^{\frac{2-n}{n}}_{\alpha^{i}\varphi_{i}}
[\int L_{g_{0}}h_{1}h_{2}
-
c_{n}n(n+2)
\sum_{i}\int\varphi_{i}^{\frac{4}{n-2}}h_{1}h_{2}
] \end{split}\end{equation*}
\end{enumerate}
\end{proposition}
\begin{proof}
Cf.  Proposition \ref{I-prop_derivatives_on_H} from \cite{may-cv}  in case 
$
H_{u}(p,\varepsilon)
=
\langle \phi_{k,i} \rangle^{\perp_{Ku^{\frac{4}{n-2}}}}.
$
 In case 
 \begin{equation*}
H_{u}(p,\varepsilon)
=
\langle \phi_{k,i} \rangle^{\perp_{L_{g_{0}}}}
 \end{equation*}
 statement  $(ii)$ still holds true by virtue of Lemma \ref{lem_comparing_orthogonalities}.  Also note, that for 
$$h\in \langle \phi_{k,i} \rangle^{\perp_{L_{g_{0}}}}
\; \text{ with }\; \Vert h \Vert=1$$ 
we have again by Lemma \ref{lem_comparing_orthogonalities}
$$
\tilde h
=
\Pi^{\top_{Ku^{\frac{4}{n-2}}}}_{\langle \phi_{k,i}\rangle}
h
=
O
(
\frac{\vert  \nabla K_{i}\vert}{\lambda_{i}}
+
\frac{1}{\lambda_{i}^{2}}
+
\frac{1}{\lambda_{i}^{n-2}}
+
\sum_{j\neq i}\varepsilon_{i,j}
+
\Vert v \Vert)$$
and hence, since $\partial J(u)=\partial J(\alpha^{i}\varphi_{i})+O(\Vert v \Vert),$
\begin{equation*}
\begin{split}
\partial J(\alpha^{i}\varphi_{i})\tilde h
= &
(\partial J(u)+O(\Vert v \Vert))\tilde h \\
= &
O
(
\frac{\vert  \nabla K_{i}\vert^{2}}{\lambda_{i}^{2}}
+
\frac{1}{\lambda_{i}^{4}}
+
\frac{1}{\lambda_{i}^{2(n-2)}}
+
\sum_{j\neq i}\varepsilon_{i,j}^{2}
+
\Vert v \Vert^{2}
+
\vert \partial J(u)\vert^{2}
).
\end{split}
\end{equation*}
Hence the Proposition  follows.
\end{proof}

\begin{proposition}\label{prop_derivatives_on_H_f}
Let $u=u_{\alpha, \beta}+\alpha^{i}\var_{i}+v\in V(\omega, p, \varepsilon)$  and 
$$h_{1},h_{2}\in H= H_{u}(\omega, p, \varepsilon).$$
Then
\begin{enumerate}[label=(\roman*)]
  \item \quad  
$
\vert \partial J(  u_{\alpha, \beta}  + \alpha^{i}\varphi_{i} )\lfloor_{H}\vert 
= 
o_{\varepsilon}(\Vert v \Vert) 
+
O(\sum_{r\neq s} \frac{\vert \nabla K_{r}\vert}{\lambda_{r}} 
+
\lambda_{r}^{\frac{2-n}{2}} +\varepsilon_{r,s}
+
\vert \partial J(u)\vert)
$
  \item \quad  and up to some $o_{\varepsilon}(\Vert h_{1} \Vert \Vert h_{2} \Vert)$ we have
 \begin{equation*}\begin{split}
\frac{1}{2}\partial^{2} & J(  u_{\alpha, \beta}  + \alpha^{i}\varphi_{i} )h_{1}h_{2}  \\
= &
k_{u_{\alpha, \beta}+\alpha^{i}\varphi}^{\frac{2-n}{n}}
[
\int L_{g_{0}}h_{1}h_{2}
-
c_{n}n(n+2)
\int 
(\frac{K\omega  ^{\frac{4}{n-2}}}{4n(n-1)}
+
\sum_{i}\varphi_{i}^{\frac{4}{n-2}}
)
h_{1}h_{2}
] 
.
\end{split}\end{equation*}
\end{enumerate}
\end{proposition}  

\begin{proof}
Cf.  Proposition \ref{I-prop_derivatives_on_H_f}  in \cite{may-cv} in case 
$
H_{u}(p,\varepsilon)    
=
\langle \phi_{k,i} \rangle^{\perp_{Ku^{\frac{4}{n-2}}}}.
$
In case 
 \begin{equation*}
H_{u}(\omega,p,\varepsilon)
=
\langle u_{\alpha,\beta},\phi_{k,i} \rangle^{\perp_{L_{g_{0}}}}
 \end{equation*}
 the statement follow from the former case via Lemma \ref{lem_comparing_orthogonalities} arguing as in the proof of Proposition \ref{prop_derivatives_on_H}.
\end{proof}


\begin{proof}[\textbf{Proof of Lemma \ref{lem_comparing_orthogonalities}}] 
Let us  just show the case 
$$ 
\nu_{1}\in H_{u}( p, \varepsilon)=\langle \phi_{k,i}\rangle
^{\perp_{Ku^{\frac{4}{n-2}}}},
$$ 
as the other cases follow analogously.
We may write with suitable  $\beta^{k,i}=O(1)$ and arbitrary $\alpha\in \R$
\begin{equation*}
\begin{split}
\Pi^{\top_{L_{g_{0}}}}_{\langle \phi_{k,i}\rangle}
\nu_{1}
= &
\beta^{k,i}\langle \nu_{1},\phi_{k,i}\rangle_{L_{g_{0}}}\phi_{k,i}
=
\beta^{k,i}\langle \nu_{1},L_{g_{0}}\phi_{k,i}\rangle_{L^{2}_{g_{0}}}\phi_{k,i} \\
= &
\beta^{k,i}\langle \nu_{1},(L_{g_{0}}-\alpha Ku^{\frac{4}{n-2}})\phi_{k,i}\rangle_{L^{2}_{g_{0}}}\phi_{k,i}.
\end{split}
\end{equation*}
From  Lemma \ref{lem_interactions}  we then find via expansion and H\"older's inequality
\begin{equation*}
\begin{split}
\int Ku^{\frac{4}{n-2}}\phi_{k,i}\nu_{1}
= &
K_{i}\int (\alpha^{j}\varphi_{j}+v)^{\frac{4}{n-2}}\phi_{k,i}\nu_{1}
+
O((\frac{\vert  \nabla K_{i}\vert}{\lambda_{i}}+\frac{1}{\lambda_{i}^{2}})\Vert \nu_{1} \Vert)
\\
= &
K_{i}\int (\alpha^{j}\varphi_{j})^{\frac{4}{n-2}}\phi_{k,i}\nu_{1}
+
O((\frac{\vert  \nabla K_{i}\vert}{\lambda_{i}}+\frac{1}{\lambda_{i}^{2}}+\Vert v \Vert)\Vert \nu_{1} \Vert).
\end{split}
\end{equation*}
Decomposing 
\begin{equation*}
M=
\{
\varphi_{i}\geq \sum_{i\neq j=1}^{p}\varphi_{j}
\}
+
\{
\varphi_{i}\leq \sum_{i\neq j=1}^{p}\varphi_{j}
\} 
\end{equation*}
and applying again Lemma \ref{lem_interactions} then show via expansion and H\"older inequality
\begin{equation*}
\begin{split}
\int Ku^{\frac{4}{n-2}}\phi_{k,i}\nu_{1}
= &
K_{i}\alpha_{i}^{\frac{4}{n-2}}\int \varphi_{i}^{\frac{4}{n-2}}\phi_{k,i}\nu_{1} \\
& +
O(
(
\frac{\vert  \nabla K_{i}\vert}{\lambda_{i}}
+
\frac{1}{\lambda_{i}^{2}}
+
\sum_{j\neq i}\varepsilon_{i,j}
+  
\Vert v \Vert)\Vert \nu_{1} \Vert),
\end{split}
\end{equation*}
where we made use of $n=3,4,5$.
Consequently
\begin{equation*}
\begin{split}
\Pi^{\top_{L_{g_{0}}}}_{\langle \phi_{k,i}\rangle}
\nu_{1}
= &
\beta^{k,i}\langle \nu_{1},(L_{g_{0}}
-
\alpha K_{i}\alpha_{i}^{\frac{4}{n-2}}
\varphi_{i}^{\frac{4}{n-2}})\phi_{k,i}\rangle_{L^{2}_{g_{0}}}\phi_{k,i} \\
& +
O(
(
\sum_{r\neq s}\frac{\vert  \nabla K_{r}\vert}{\lambda_{r}}
+
\frac{1}{\lambda_{r}^{2}}
+ 
\varepsilon_{r,s}
+
\Vert v \Vert)\Vert \nu_{1} \Vert). 
\end{split}
\end{equation*}
Note, that $L_{g_{0}}\phi_{k,i}=c_{k}\varphi_{i}^{\frac{4}{n-2}}\phi_{k,i}$ on $\R^{n}$ for suitable constants $c_{k}$, while  
\begin{equation*}
\Vert L_{g_{0}}\phi_{k,i}-c_{k}\varphi_{i}^{\frac{4}{n-2}}\phi_{k,i})\Vert_{L^{\frac{2n}{n+2}}}
=
O(\frac{1}{\lambda_{i}^{2}}+\frac{1}{\lambda_{i}^{n-2}})
\end{equation*}
generally, cf. Lemma \ref{II-lem_emergence_of_the_regular_part} in \cite{MM1}.  Hence choosing $\alpha$ suitably, we derive 
\begin{equation*}
\Pi^{\top_{L_{g_{0}}}}_{\langle \phi_{k,i}\rangle}
\nu_{1}
=
O(
(
\sum_{r\neq s}\frac{\vert  \nabla K_{r}\vert}{\lambda_{r}} 
+
\frac{1}{\lambda_{r}^{2}}
+
\frac{1}{\lambda_{r}^{n-2}}
+
\varepsilon_{r,s}
+
\Vert v \Vert)\Vert \nu_{1} \Vert),
\end{equation*}
what had to be shown.    
\end{proof}

\end{document}